\documentclass[a4paper,11pt,english]{smfart}
\usepackage[cp1250]{inputenc}
\usepackage{a4wide}
\usepackage{amsmath,amssymb}
\usepackage{amsthm}
\usepackage{graphicx}
\usepackage{epstopdf}
\usepackage{xcolor}
\usepackage{multirow}
\usepackage{tabularx}
\def\di{\displaystyle}
\def\N{\mathbb{N}}
\def\R{\mathbb{R}}

\def\L{\mathcal{L}}

\newtheorem{df}{Definition}
\newtheorem{lem}{Lemma}
\newtheorem{twr}{Theorem}
\newtheorem{rem}{\textbf{Remark}}

\begin{document}
\title[Discrete and continuous fractional persistence problems]{Discrete and continuous fractional persistence problems -- the positivity property and applications}
\author{Jacky Cresson and Anna Szafra\'{n}ska}
\date{}
\maketitle

\begin{abstract}
In this article, we study the continuous and discrete fractional persistence problem which looks for the persistence of properties of a given classical ($\alpha=1$) differential equation in the fractional case (here using fractional Caputo's derivatives) and the numerical scheme which are associated (here with discrete Grünwald-Letnikov derivatives). Our main concerns are positivity, order preserving ,equilibrium points and stability of these points. We formulate explicit conditions under which a fractional system preserves positivity. We deduce also sufficient conditions to ensure order preserving. We deduce from these results a fractional persistence theorem which ensures that positivity, order preserving, equilibrium points and stability is preserved under a Caputo fractional embedding of a given differential equation. At the discrete level, the problem is more complicated. Following a strategy initiated by R. Mickens dealing with non local approximations, we define a non standard finite difference scheme for fractional differential equations based on discrete Grünwald-Letnikov derivatives, which preserves positivity unconditionally on the discretization increment. We deduce a discrete version of the fractional persistence theorem for what concerns positivity and equilibrium points. We then apply our results to study a fractional prey-predator model introduced by Javidi and al.
\end{abstract}

\noindent



\tableofcontents

\newpage
\part{Introduction}

In recent years, many models describing biology and epidemiology phenomena concerning population dynamics, dynamical behaviour of epidemic diseases, etc make use of fractional derivatives. These systems are most of the time fractional generalization of classical models, meaning that the classical derivative is directly replaced by a chosen notion of fractional derivative (Riemann-Liouville, Caputo, etc). Such kind of generalization can be described using the {\it fractional embeddings} formalism for ordinary differential equations introduced for example in \cite{cresson07,cresson-inizan,cresson13}. Informally, let us consider a differential equation of the form
\begin{equation}
\label{equa}
\di\frac{dx}{dt} =f(x) ,\ x\in \R^n ,\ n\in \N .
\end{equation}
We denote by $\mathcal{D}^{\alpha}$ a given choice of a fractional derivative (Riemann-Liouville, Caputo, etc). The fractional embedding of equation (\ref{equa}) is given by
\begin{equation}
\label{equaemb}
\di\mathcal{D}^{\bar{\alpha}} x =f(x) ,\ x\in \R^n ,\ n\in \N ,
\end{equation}
where $\bar{\alpha} =(\alpha_1 ,\dots ,\alpha_n)$, $x=(x_1, \do ts ,x_n)$ and $\mathcal{D}^{\bar{\alpha}} x =(\mathcal{D}^{\alpha_1 } x_1 ,\dots , \mathcal{D}^{\alpha_n } x_n )$.\\

In such a context, many natural questions arise which all deal with the persistence of some properties of the initial classical system under a fractional embedding. Even if the problem is first formulated in the continuous setting, the need for a numerical exploration of these generalizations will also produce a discrete analogue of this persistence problem.

\section{Fractional persistence problem : the continuous case}

The need for a fractional generalization of a given classical model is often due to new behaviors which can not be taken into account by the model. It is the case for example, for the mathematical model used for the dynamics of a dengue fever epidemic (see \cite{pooseh},\cite{diethlem}). In this situation, it can be useful to look for a fractional deformation of the initial system, trying to fit the fractional exponent of differentiation in order to catch properly the data (see for example \cite{diethlem}, $\S$.3.2, p.617). During this procedure however, one must be very careful. Indeed, the initial system possess specific properties (range of value for the variables, symmetries, first integrals, variational structures, etc) which come from the phenomenon itself and have nothing to do with the underlying mathematical framework used to model it. This is in particular the case in biology and more specifically in population dynamics when dealing with densities of populations : each variable must remain between $0$ and $1$. As a consequence, we are lead to the following problem :\\

{\bf Fractional persistence problem} : {\it Assume that the underlying classical ODE (i.e. $\alpha =1$) satisfies a set of properties $\mathcal{P}$. Under which conditions the fractional version satisfies also properties $\mathcal{P}$ ?}\\

Previous results on this problem was derived for the persistence of symmetries and Noether type theorem (see \cite{cresson07},\cite{bourdin-cresson2}) and also the persistence of a variational structure and more precisely a Lagrangian or Hamiltonian one (see \cite{cresson-inizan},\cite{cresson-inizan2}). \\

In this paper, we {\bf answer the persistence problem under fractional embedding for} what concerns {\bf the positivity property and the stability in the Caputo setting}. Our result is limited to a large class of differential equations which arises naturally from the necessary and sufficient conditions for positivity of the classical underlying differential equation. We recover in particular under natural assumptions a class introduced by Dimitrov and Kojouharov in \cite{dk} and containing a large class of population dynamics models. \\

Previous results in this direction have been obtained and discussed by numerous authors but up to our knowledge, none of these works cover our result. Precisely, we have found two kind of works related to different strategies of proof.\\

\begin{itemize}
\item The first series of articles by Vasundgara Devi  et al.  in \cite{vasun} and Girejko et al. \cite{girejko} deal with a generalization of the classical proof of the Nagumo-Brezis theorem in the classical Viability theory. The article by Girejko et al. \cite{girejko} provide explicit conditions for positivity (see Corollary 10,p.16) which reduce to the classical one. However, it seems that unfortunately the proof of their result is not complete (see \cite{caja}). The reason is that the generalization of the tangency condition in the fractional case is not correct. The result is corrected in \cite{caja} but no explicit conditions which can be easily check from the data of the fractional systems are given. Then, our result is not a consequence of these results and even our explicit conditions seems to be very difficult to deduce from the abstract tangency condition obtained in \cite{caja}.

\item More specific results exist in the literature, but the proofs which are given, in particular the one related to a {\it generalized version of the mean value theorem} \cite{trujillo} in the fractional case (see \cite{parra}, $\S$.4, Theorem 4.2 p. 2221 or (\cite{goufo},Property ii) p.4), (\cite{rihan},Theorem 6 p.4 and Remark 3 p.2), (\cite{arafa},Theorem 2 p.541)) are incomplete. In fact, these results always assume that solution of the fractional differential equations are such that the solution is continuous as well as the fractional derivative of this solution. This point is never proved or discussed in the previously cited article
\end{itemize}

\section{Fractional persistence problem : the discrete case}

Having constructed a viable fractional system satisfying the basic properties, the basic problem is to study its dynamical behavior. However, as in the classical case and even more difficult, it is usually not possible to solve the fractional differential equations and to provide explicit solutions. As a consequence, we are leaded to a numerical study of these equations. There exists already some numerical scheme in the literature. The most popular is the famous {\it Gr\"uwald-Letnikov scheme} (see for example \cite{scherer} or \cite{pod}). When no fundamental properties of the fractional model need to be satisfied, then the classical Gr\"unwald-Letnikov scheme gives, for a sufficiently small $h$, results which are in good agreement with the expected behavior of the system (see \cite{scherer}).

In general, {\bf simulations} are {\bf used to validate a given continuous model}. When this model satisfies fundamental properties like positivity, stability, etc, then one must be sure that the numerical scheme preserves these properties. {\bf The problem is that the convergence of a given numerical scheme is not sufficient to ensure the persistence of these properties at the discrete level}. For example, the classical Gr\"unwald-Letnikov scheme does not preserve positivity. So, the simulations that we obtain in this case can be completely unrealistic or producing artifacts with respect to the continuous model. We then formulate the following discrete version of the fractional persistence problem :\\

{\bf Discrete fractional persistence problem} : {\it Assume that the underlying classical ODE (i.e. $\alpha =1$) satisfies a set of properties $\mathcal{P}$ and that these properties are preserved under the fractional embedding of the equation. Can we construct a discrete numerical scheme for the fractional equation such that the discrete analogue of the properties $\mathcal{P}$ are satisfied ?}\\

This problem is already known in the study of classical differential equations and is an active field of research. We refer in particular to the work of Ronald E. Mickens (see for example \cite{mic3}) which has designed during the 80's new finite differences scheme to preserve dynamical properties of classical differential equations which he called {\it non-standard methods}.\\

In this paper, we adapt the strategy of R. Mickens to the fractional case in the context of the positivity property. We {\bf define a non standard finite difference scheme for our class of fractional differential systems which preserves positivity}. In particular, we {\bf prove} the {\bf convergence} of our scheme and give a comparison with results obtain by the classical Gr\"unwald-Letnikov scheme.\\

Several authors have proposed numerical scheme for fractional differential equations. For example, and the list is not exhaustive, we refer to \cite{scherer},\cite{pod},\cite{delfim},\cite{ata1},\cite{ata2}, etc. However, most of these schemes are not proved to be convergent. Moreover, up to our knowledge, {\bf none of these numerical schemes preserve positivity} or stability.

\section{Application : a fractional predator-prey model}

In \cite{JN}, the authors introduce a fractional predator-prey model and provide some numerical simulations. Numerous problems appears during this study.\\

\begin{itemize}
\item The fractional predator-prey model is a fractional embedding of a classical predator-prey model defined by S. Chakraborty and al. in \cite{chakra} which generalizes H. Freedman's two dimensional model with one carrying capacity (see \cite{freedman},Chap.7,$\S$.7.3,p.147). The two models, classical and fractional, are not proved to satisfy the positivity property. However, as these models deal with population densities, they need to satisfy the positivity property in order to be viable.

\item In \cite{JN}, the authors use a particular numerical scheme due to Atanackovic and Stankovich in (\cite{ata1},\cite{ata2}). This method is not proved to be convergent and deserves more study as already pointed out in \cite{ata2}.

\item Even assuming the convergence of the method, the numerical scheme produces for small value of the time step increment, numerical results which are not in accordance with the expected theoretical behaviour. In particular, the stability of the equilibrium points and positivity of the solutions are not preserved.
\end{itemize}

In this paper, using our results, we {\bf prove that the fractional predator-prey model satisfies the positivity property}, allowing us to {\bf ensure the viability of the model}. Moreover, {\bf applying our numerical scheme}, we obtain a {\bf very good agreement} with the expected {\bf theoretical behavior}. We also perform several simulations in order to study the robustness of our scheme. We obtain numerical evidences that {\bf our numerical scheme} produce {\bf simulations equivalent to} the one obtained via the {\bf Gr\"unwald-Letnikov} method but with a {\bf time step increment} at least {\bf ten times bigger}.

\section{Organization of the paper}

The plan of the paper is the following. In Section \ref{remind} we remind some classical definition about the Caputo fractional derivatives and the discrete Gr\"unwald-Letnikov fractional derivatives. We also recall classical results about existence and unicity of solutions for the Caputo Cauchy problem. In Section \ref{class} we introduce our class of fractional differential systems generalizing classical models in biology and medicine. We prove in particular that these systems satisfy the positivity property. In Section \ref{scheme}, we define a non standard finite difference scheme for our class of fractional differential systems and prove its convergence as well as the fact that its preserves positivity. We also gives a comparison with the results obtain using a classical Gr\"unwald Letnikov method on a toy model. The last Section is devoted to the numerical study of the fractional prey-predator model introduced in \cite{JN}. We focus on some specific problems related to the simulation provided in \cite{JN} using the numerical scheme of Atanackovic and Stankovich in \cite{ata2}. In particular, we prove that in all these case, our numerical scheme behaves very nicely and for a time step increment at least ten times bigger as what is need in their computations.




\newpage
\part{Continuous fractional persistence problems - the positivity property}
\setcounter{section}{0}

\section{Reminder about fractional differential equations}
\label{remind}
This section contains the basic definitions and properties in the theory of the fractional calculus.

\subsection{Caputo's fractional derivatives}

In first introduce the following notations. Let $(a,b)\in \R^2$ such that $a<b$. We denote by $L^1$ the usual Lebesgue space and by $AC:=AC ([a,b],\R^n )$ the space of absolutely continuous functions on $[a,b]$.\\

The {\it left fractional Riemann-Liouville integrals} of order $\alpha>0$ of $x\in L^1$ is defined by
\begin{equation}
I_{a+}^{\alpha} [x] (t) :=\di\frac{1}{\Gamma (\alpha )} \di\int_a^t \di\frac{x(s)}{(t-s)^{1-\alpha}} ds,\ t\in [a,b]\ q.e.
\end{equation}

Recall that a function $x\in AC$ if and only if there exists a couple $(c,\phi )\in \R^n \times L^1$ such that
\begin{equation}
\label{ac-cond}
x(t)=c+I_{a+}^1 [\phi ](t),\ t\in [a,b].
\end{equation}
In this case, we have $c=x(a)$ and $\phi(t)=\dot{x}(t)$, $t\in [a,b]$ a.e.

\begin{df}
The Caputo differential operator of order $\alpha>0$ is given by
\begin{equation}
\label{cap}
_{c}D_{a+}^{\alpha}x(t)=\frac{1}{\Gamma(m-\alpha)}\int_{a}^{t}(t-s)^{m-\alpha-1}\frac{d^m}{dt^m}x(s)ds,
\end{equation}
where $\Gamma(\cdot)$ is the Euler gamma function and $m=\lceil \alpha\rceil$ defines the smallest integer larger then $\alpha$.
\end{df}

The Caputo derivative exists if and only if $x\in AC$. In that case, one obtain the following representation
\begin{equation}
_{c}D_{a+}^{\alpha} [x](t)=I_{a+}^{1-\alpha} [\phi ] ,
\end{equation}
where $x$ is identified with a couple $(x(a),\phi)$ satisfying (\ref{ac-cond}).

\subsection{Existence and regularity of solutions}

Let us consider the general fractional autonomous differential equation with Caputo derivative (\ref{cap})
\begin{equation}
\label{eq}
\di _{c}D^{\alpha}_{a+} [x] =f(x),
\end{equation}
where $f:\R^n\to\R^n$. This class of equation covers most of the known fractional models used in applications to biology.\\

A {\it local solution} to the Cauchy problem is the data of a function $x$ such that $\di _{c}D^{\alpha}_{a+} [x]$ exists and satisfies equation (\ref{eq}).\\

A result from K. Diethlem  gives some conditions which ensure the existence of local solutions for the Cauchy problem.

\begin{twr}
Let $f$ be locally Lipschitz. Then there exists a unique local solution to the Cauchy problem (\ref{eq}).
\end{twr}

The condition on $f$ to be locally Lipschitz is certainly not optimal. However, all the examples that we have in the following are satisfying this property. For more general class of functions in the context of {\it weak solutions}, we refer to the work of L. Bourdin \cite{bo}.\\

A useful regularity property of the solutions is :

\begin{twr}
\label{regularity}
Let $f$ be locally Lipschitz and $x$ be a local solution of the Cauchy problem. Then, we have $x\in AC$ and $\di _{c}D^{\alpha}_{a+} [x] \in C^0$.
\end{twr}

\begin{proof}
If $x$ is a solution of the fractional differential equation then it possesses a fractional Caputo derivative $\di _{c}D^{\alpha}_{a+} [x]$. This implies that $x\in AC$ and in particular that $x\in C^0$. As $f$ is locally Lipschitz, we deduce that $f(x) \in C^0$ and as a consequence $\di _{c}D^{\alpha}_{a+} [x] \in C^0$.
\end{proof}

This Theorem is precisely the missing argument in all the proofs that we have found using the generalized mean value theorem stated in \cite{trujillo} (see below).

\section{Linear stability and equilibrium points}

\subsection{Equilibrium points}

A useful property of the Caputo fractional embedding of differential equations is that it preserves exactly the set of equilibrium points. Indeed, we have :

\begin{df}
A point $x_0$ is an equilibrium point of (\ref{eq}) if and only if the solution of of the Cauchy problem with $x(0)=x_0$ is such that $x(t)=x_0$ for all $t\geq 0$.
\end{df}

A useful characterization of equilibrium points is given by :

\begin{lem}
A point $x_0$ is an equilibrium point of (\ref{eq}) if and only if $f(t,x_0)=0$ for all $t\geq 0$.
\end{lem}

As this characterization is not related to the left-hand side of the equation, we deduce :

\begin{lem}
\label{persistence-equilibrium}
The set of equilibrium points is exactly preserved under a Caputo fractional embedding.
\end{lem}

\subsection{Linear stability}

Usually, the dynamics of the classical differential equation is well studied for what concerns the equilibrium points and their stability. A natural question is then to see if the stability nature of a given equilibrium point is preserved under a Caputo fractional embedding. In the following, we restrict ourself to {\it linear stability}.\\

We use the definition of {\it Lyapounov stability} or simply stability of an equilibrium point. In the following, we denote by $B(x,r)$ the open ball centered at $x$ of radius $r$, that is the set $\{ y\in \mathbb{R}^n ,\ \parallel y-x\parallel <r \}$ where $\parallel \cdot \parallel$ denotes the canonical norm on $\mathbb{R}^n$.

\begin{df}
An equilibrium point $x_0$ is said to be stable if for all $\epsilon >0$, there exists a $\delta >0$ such that
for all $x\in B(x_0 ,\delta )$ the solution $\phi (t,x)$ stays in $B(x_0 ,\epsilon )$ for all $t\geq 0$. If moreover, we have $\lim_{t\rightarrow +\infty} \phi (t,x)=x_0$ the equilibrium point is said to be asymptotically stable.
\end{df}

The problem of the stability of an equilibrium point for a non linear Caputo fractional differential equation is difficult. We refer to \cite{pod} for some results in this direction. For linear systems however, the situation is simple and well understood. In particular, we have the following well known result due to D. Matignon \cite{mat}.

\begin{twr}
Let us consider a linear Caputo fractional differential equation
\begin{equation}
_{c}D^{\alpha}_{a+} [x] = A.x ,
\end{equation}
where $A$ is a square matrix of dimension $n$. The origin is stable if and only if the eigenvalues of $A$, denoted by $\lambda_i$, $i=1,\dots ,n$ satisfy the condition
$$
\mid \mbox{\rm arg} (\lambda_i ) \mid > \alpha \di\frac{\pi}{2} ,
\eqno{(S)_{\alpha}}
$$
where $\mbox{\rm arg} (\lambda )$ denotes the argument of $\lambda \in \mathbb{C}$.
\end{twr}

It must be noted that for $\alpha=1$, one recover the classical result.\\

The concept of {linear stability} for an equilibrium point used the notion of {\it linearized equation} at a given equilibrium point $x_0$.

\begin{df}
Let $x_0$ be an equilibrium point of the Caputo fractional equation (\ref{eq}). The linearized equation associated to (\ref{eq}) at the equilibrium point $x_0$ is defined as
\begin{equation}
_{c}D^{\alpha}_{a+} [x] =Df (t,x_0) .x ,
\end{equation}
where $Df (x_0)$ is the differential of $f$ at $x_0$.
\end{df}

We have :

\begin{df}[Linear stability] Let $x_0$ be an equilibrium point of the Caputo fractional equation (\ref{eq}). The point $x_0$ is said to be linearly stable (resp. unstable) if the origin is a stable (resp. unstable) equilibrium point of the associated linearized equation at $x_0$.
\end{df}

The previous Theorem gives a necessary and sufficient condition for an equilibrium point to be linear stable or unstable. Moreover, as we are considering the case $0<\alpha <1$, we have :

\begin{twr}
\label{persistence-stability}
The Caputo fractional embedding of a differential equation preserves the linear stability nature of the equilibrium points.
\end{twr}

The information about the stability/instability of the equilibrium points for the classical system is usually known. As a consequence, one can directly deduce from Theorem \ref{persistence-stability} the corresponding result in the fractional case.

\begin{rem}
In many article, there exists some confusion between linear stability and local stability. The local stability deals with the stability of the solution in a neighborhood of the equilibrium point. In \cite{sadar} for example, the authors asserts that the equilibrium of their model are locally stable. They use a result of Tavazoei and al. \cite{tavazoei} about local stability of Caputo fractional differential equations. However, the notion of stability used in \cite{tavazoei} is the linear one. The same is true for other articles (see in particular \cite{arafa}).
\end{rem}

\section{The positivity problem for Caputo fractional differential equations}

In this Section, we discuss the positivity problem and we give an explicit criterion which ensures that a given Caputo fractional differential equation satisfies the positivity property.

\subsection{The positivity property}

As already discussed in the Introduction, many models are dealing with quantities which must remain positive during the time evolution of the system. This is the case for example in population dynamics where the variable are associated to densities of population. \\

{\bf Positivity property} : {\it A fractional differential system of the form (\ref{eq}) satisfies the positivity property is for all initial conditions $x_0 \in \mathbb{R}_+^m$, the solution $\phi_t (x_0 )$ passing through $x_0$ at time $t=0$ remains positive for all $t>0$.}\\

The {\it positivity problem} is to find explicit necessary and sufficient conditions $(P)_{\alpha}$ under which a system satisfies the positivity property. When $\alpha=1$, the positivity problem is completely solved. We have (see \cite{pavel,walter}) :

\begin{twr}
\label{posi}
A system of the form (\ref{eq}) with $\alpha =1$ satisfies the positivity property if and only if for all $i=1,\dots ,m$, $f_i (x)\geq 0$ for all $x\in \mathbb{R}_+^m$ such that $x_i=0$ $(P)_1$.
\end{twr}

For $0<\alpha <1$, the problem seems to be more difficult. Up to our knowledge, this problem was studied first in Vasundgara Devi et al.  in \cite{vasun} and Girejko et al. \cite{girejko}  in which some explicit sufficient (but a priori also necessary) conditions very similar to conditions (P) are obtained (see \cite{girejko},Corollary 10). However, it seems that the proof of these results are incomplete as pointed out in \cite{caja}. These authors obtain sufficient abstract conditions ensuring positivity and more generally viability. However, it seems not easy to deduce from their conditions an explicit one.\\

Another approach used the generalized mean value Theorem proved in \cite{trujillo} as for example in (\cite{parra}, $\S$.4, Theorem 4.2 p. 2221) or (\cite{goufo},Property ii) p.4), (\cite{rihan},Theorem 6 p.4 and Remark 3 p.2), (\cite{arafa},Theorem 2 p.541)). However, all these proofs are incomplete. They proved that under some assumptions on the solution and its fractional derivative, then one can prove the positivity. They never prove that these assumptions are satisfied by their model. Following the same strategy, and completing the argument, we prove that :

\begin{twr}
\label{thm-positivity}
Let $f$ be locally Lipschitz. Assume that $f$ satisfies condition $(P)_1$ then the fractional differential equation satisfies the positivity property.
\end{twr}

\begin{proof}
The proof is based on the following result of J.J. Trujillo and al. \cite{trujillo} called the {\it generalized mean value theorem} :

\begin{twr}[Generalized mean value theorem]
\label{mean}
Let $0<\alpha \leq 1$, $x \in C^0$ and $_{c}D^{\alpha}_{a+} [x] \in C^0$. Then, we have
\begin{equation}
x(t)=x(a)+\di\frac{1}{\Gamma (\alpha )} _{c}D^{\alpha}_{a+} [x] (s) (t-a)^{\alpha} ,
\end{equation}
where $a\leq s \leq t$ and for all $t\in [a,b]$.
\end{twr}

Let us consider a solution of the Caputo fractional differential equation $x$. By the regularity result obtained in Theorem \ref{regularity}, we deduce that the assumptions of Theorem \ref{mean} are satisfied.

Assume that there exists $t_1$ such that $x(t_1 )< 0$. We denote by $t_2$ the minimal value of $0< t<t_2$ such that $x(t_2 )<0$ and $x(t) \geq 0$ for $t < t_2$. By the generalized mean value Theorem, we have for $t\leq t_2$, that
\begin{equation}
x(t)=x(0)+\di\frac{1}{\Gamma (\alpha )} f(x(s)) t^{\alpha} ,
\end{equation}
with $x(0) \geq 0$ and $f(x(s))\geq 0$ by assumption $(P)_1$ as $x(s) \geq 0$. We deduce that $x(t_2) \geq 0$ in contradiction with our assumption. This concludes the proof.
\end{proof}

\subsection{A fractional comparison theorem}

As usual with a positivity Theorem, one can deduce sufficient conditions in order to ensure a comparison Theorem. In the classical case, a necessary and sufficient condition is known.\\

A function $f :\R^n \times \R \rightarrow \R^n$ is said to be {\it quasi-monotone} if and only if for all $i=1,\dots ,n$ and $(x,y)\in \R^n$ such that $x_j \geq y_j$, $j=1,\dots ,n$, $j\not= i$ and $x_i =y_i$ we have
\begin{equation}
\label{monotone}
f_i (x,t) \geq f_i (y,t) .
\end{equation}

We have the classical result :

\begin{twr}
The classical differential equation (\ref{frac}) with $\alpha=1$ satisfies a comparison principle if and only if $f$ is quasi-monotone.
\end{twr}

We then consider $x$ and $y$ two solution of a fractional Caputo differential equation. Let us assume that $x (a) \geq y(a)$. We denote by $w =x-y$ then
\begin{equation}
\label{frac-modif}
_{c}D^{\alpha}_{a+} [w]=f(x,t)-f(y,t) ,
\end{equation}
with $w(a)=x(a)-y(a) \geq 0$. In order for the fractional system (\ref{frac}) to be order preserving, we must ensure that the fractional system (\ref{frac-modif}) preserves the positivity, i.e. that if $w(a) \geq 0$ then $w(t) \geq 0$ for $t>0$. By Theorem \ref{thm-positivity}, this is the case if the function $g(w,t)=f(x,t)-f(y,t)$ satisfies conditions $(P)_1$. Then, for all $i=1,\dots ,n$ and $w \geq 0$ such that $w_i =0$, we must have $g(w,t) \geq 0$. As $w=x-y$, this condition reduces exactly to the quasi-monotonicity of $f$. We then have proved :

\begin{twr}[Fractional comparison theorem]
Let $f$ be locally Lipschitz and $f$ quasi-monotone, then the fractional system (\ref{frac}) satisfies the comparison principle.
\end{twr}

The interesting fact is that the conditions ensuring the positivity or the comparison principle are similar in the fractional and classical setting.

\subsection{A class of fractional differential equations}
\label{class}

In this Section, we introduce a class of fractional differential equations which is in fact generic when one is dealing with systems preserving positivity. Precisely, we have :

\begin{twr}[Representation theorem]
\label{representation}
Caputo fractional differential equations of the form (\ref{eq}) with $f\in C^1$ and satisfying condition $(P)_1$ can be written as
\begin{equation}
\label{frac}
\di _{c}D^{\alpha}_{a+} [x] =f_+(x) -xf_- (x),
\end{equation}
where $f_+$ and $f_-$ are two applications from $\R^m$ to $\R^m$ which are positive when $x\geq 0$ and Lipschitz continuous.
\end{twr}

The previous Theorem implies that we can restrict our attention to fractional differential equations of the form (\ref{frac}) when one is dealing with a Caputo fractional embedding of a classical equation satisfying the positivity property. Indeed, in this case, the condition $(P)_1$ is a necessary and sufficient condition for positivity so that the associated fractional equation directly satisfies the conditions of Theorem \ref{representation}. \\

When $\alpha=1$ and $m=2$, equations of the form (\ref{frac}) correspond to models studied by Dimitrov and Kojouharov \cite{dk} which contains classical examples like the {\it general Rozenzweig-MacArthur predator -prey model} (see \cite{bc},p.182). When $0<\alpha  <1$, this contains the fractional predator-prey model studied by Javidi and Nyamoradi in \cite{JN}.

\begin{proof}
As $f=(f_1 ,\dots ,f_n )$ is $C^1$, we can represent each $f_i$, $i=1,\dots ,n$ as
\begin{equation}
f_i (x)=f_i (x_1 ,\dots , 0 ,\dots ,x_n ) +x_i \di\int_0^1 \di\frac{\partial f_i}{\partial x_i} (x_1 ,\dots , sx_i ,\dots ,x_n )\, ds ,
\end{equation}
that is
\begin{equation}
f_i (x)=f_i (x_1 ,\dots , 0 ,\dots ,x_n ) +x_i g_i (x) ,
\end{equation}
with a bounded function $g_i :\R^n \rightarrow \R$.

Let us denote by $g_+ (x)=\max (g(x),0)$ and $g_- (x)=-\min (g(x),0)$ so that $g(x)=g_+ (x)-g_- (x)$. Using this decomposition, we obtain
\begin{equation}
f_i (x)=\left ( f_i (x_1 ,\dots , 0 ,\dots ,x_n ) + x_i g_{+,i} (x) \right ) -x_i g_{-,i} (x) .
\end{equation}
Denoting by $f_{+,i} (x) =f_i (x_1 ,\dots , 0 ,\dots ,x_n ) + x_i g_{+,i} (x)$ and $f_{-,i} (x) =g_{-,i} (x)$, we obtain the required form of the equation. Moreover, as $f$ satisfies condition $(P)_1$, we have that
$f_i (x_1 ,\dots , 0 ,\dots ,x_n ) \geq 0$ for $x\geq 0$ and as a consequence $f_{+,i} (x)$ is a positive function for $x\geq 0$.

The $\max$ and $\min$ functions preserve Lipschitz continuity. As $f$ is assume to be $C^1$, we deduce that $f_+$ and $f_-$ are at least Lipschitz continuous functions.
\end{proof}

\section{The continuous persistence problem}

All the previous results can be summarized in a unique theorem about persistence of some properties of a classical differential equation under a Caputo fractional embedding. Precisely, we have :

\begin{twr}[Fractional persistence theorem] \label{FPT}
The Caputo fractional embedding of an ordinary differential equation preserves the following properties :
\begin{itemize}
\item Exactly the set of of equilibrium points.

\item The linear stability of the equilibrium points.

\item Positivity.

\item Order preserving.
\end{itemize}
\end{twr}

As a consequence, studying some fractional version of a well-known differential equations, one can reduce the number of computations to minimal using already known properties of the classical system.



\newpage
\part{Discrete fractional persistence problems - a NSFD positivity preserving scheme}
\label{scheme}
\setcounter{section}{0}

In this Part, we introduce a non standard finite difference scheme for Caputo fractional differential equations of the form
\begin{equation}
\di _{c}D^{\alpha}_{a+} [x] =f_+(x) -xf_- (x),
\end{equation}
where $f_+$ and $f_-$ are two applications from $\R^n$ to $\R^n$ which are positive when $x\geq 0$ and Lipschitz continuous. As proved in Theorem \ref{representation}, this class of fractional systems is generic when dealing with systems preserving positivity.

\section{Grünwald-Letnikov (GL) fractional derivatives and the GL scheme}

\subsection{Discrete Grünwald-Letnikov fractional derivative}

Let $\alpha >0$, we denote by $(\alpha )_r$ the quantities defined by
\begin{equation}
\label{lam}
\alpha_0^{[\alpha]}=1,\;\;\; \alpha_j^{[\alpha]}=\prod_{k=0}^{j-1}\frac{k-\alpha}{k+1} .
\end{equation}

The Grünwald-Letnikov fractional derivative is defined as :

\begin{equation}
\alpha^{[\alpha]} _0 =0,\ \ \alpha^{[\alpha]}_r =\alpha \cdot (\alpha -1) \cdots (\alpha -r +1)/r!,\;\;\; r\in \mathbb{N}^* ,
\end{equation}
where $\mathbb{N}^*$ is a set of all natural numbers larger than zero. The Grünwald-Letnikov fractional derivative is defined as :

\begin{df}
Let $x : [a,b] \rightarrow \mathbb{R}^m$. The left fractional derivative of Grünwald-Letnikov with inferior limit $a$ of order $\alpha >0$ of $x$ is defined by
\begin{equation}
\forall t\in ]a,b],\  {}_{GL} D^{\alpha}_{a+} [x] (t)= \lim_{h\rightarrow 0^+ , rh =t-a} \frac{(-1)^r}{r!} \alpha^{[\alpha]}_r x(t-rh) ,
\end{equation}
provided that the right-hand side is well defined.
\end{df}

The Grünwald-Letnikov and Riemann-Liouville derivatives coincide for $x\in C^{[\alpha ] +1} ([a,b] ,\mathbb{R}^m )$ where $[\alpha ]$ denotes the floor of $\alpha >0$. This leads to the following discrete approximation of the Caputo derivative :

\begin{df}
Let $x\in C(\mathbb{T} ,\mathbb{R}^m )$, when $\mathbb{T}$ is a fixed time scale, i.e. subinterval of $\mathbb{R}$. The left discrete fractional derivative of Caputo-Grűnwald-Letnikov of inferior limit $a$ of order $0<\alpha <1$ is defined by
\begin{equation}
\forall k+1,\dots ,N-1,\ {}_c \Delta_{a+} [x] (t_k)= \frac{1}{h^{\alpha}} \sum_{r=0}^k
\frac{(-1)^r}{r!} \alpha^{[\alpha]}_r [x-x(a)](t_{k-r} ) .
\end{equation}
\end{df}

In order to prove the convergence of our numerical scheme, we need the following well known result about the order of approximation of the discrete Caputo-Grünwald-Letnikov derivative :

\begin{twr}
Let $x$ be a $C^1$ function and $0<\alpha <1$. We have the following inequality
\begin{equation}
\label{est}
_{c}D_{0,t}^{\alpha}x(t)=\frac{1}{h^{\alpha}}\sum_{k=0}^n \alpha_{k}^{[\alpha]}(x(t-kh)-x(0))+\mathcal{O}(h).
\end{equation}
\end{twr}

As a consequence the discrete Caputo-Grünwald-Letnikov fractional derivative is an approximation of order 1 of the Caputo fractional derivative.

\subsection{The Grünwald-Letnikov scheme}

Let $N \in \N^*$, and $X=(x_1 ,\dots ,x_n )$. We denote by ${}_c \Delta_{a+} [X]$ the quantity defined by
\begin{equation}
{}_c \Delta_{a+} [X]_i = \di\frac{1}{h^{\alpha}} \di\sum_{r=0}^i
\frac{(-1)^r}{r!} \alpha^{[\alpha]}_r [x_{i-r}-x_0] ,\ \ \ i=0,\dots ,N-1.
\end{equation}

The {\it Grünwald-Letnikov} scheme is then defined by
\begin{equation}
{}_c \Delta_{a+} [X] = f(X) ,
\end{equation}
meaning that for all $i=0,\dots ,N-1$, we have ${}_c \Delta_{a+} [X]_i = f(X_i )$.\\

For $\alpha =1$ we recover the classical Euler scheme.

\section{Classical problems related to the Grünwald-Letnikov scheme}

In this Section we illustrate some problems which are related to the use of the Grünwald-Letnikov numerical scheme. This scheme is currently used in many applications about population dynamics or medicine like epidemiology. In all these applications, most of the variables must remains positive and sometimes bounded in some domains. Moreover, the existence of equilibrium points as well as the stability of these points is of course essential for the understanding of the model. The Grünwald-Letnikov scheme does not respect most of them and one must be very careful by choosing the time step of integration in order to be sure that we observe the correct behavior. Here, we give several examples of these problems for which our scheme will provide a better answer.\\

We study a class of toy model given by
\begin{equation} \label{exmodel}
\begin{array}{lll}
{}_c D^{\alpha} x  & = & bx(1-x)-a\frac{xy}{1+x}, \\
{}_c D^{\alpha} y & = & \frac{xy}{1+x}-cy,
\end{array}
\end{equation}
where $a,b,c$ are real constants.\\

We use two particular sets of values for our simulations :
\begin{itemize}
\item {\bf Model 1} : $a=2,b=1,c=6$.
\item {\bf Model 2} : $a=2,b=1,c=0.2$.
\end{itemize}

Expected dynamical behavior for the above models is presented on Figure 1.

\begin{center}
\begin{figure}[ht]
\centerline{%
    \begin{tabular}{cc}
        \includegraphics[width=0.48\textwidth]{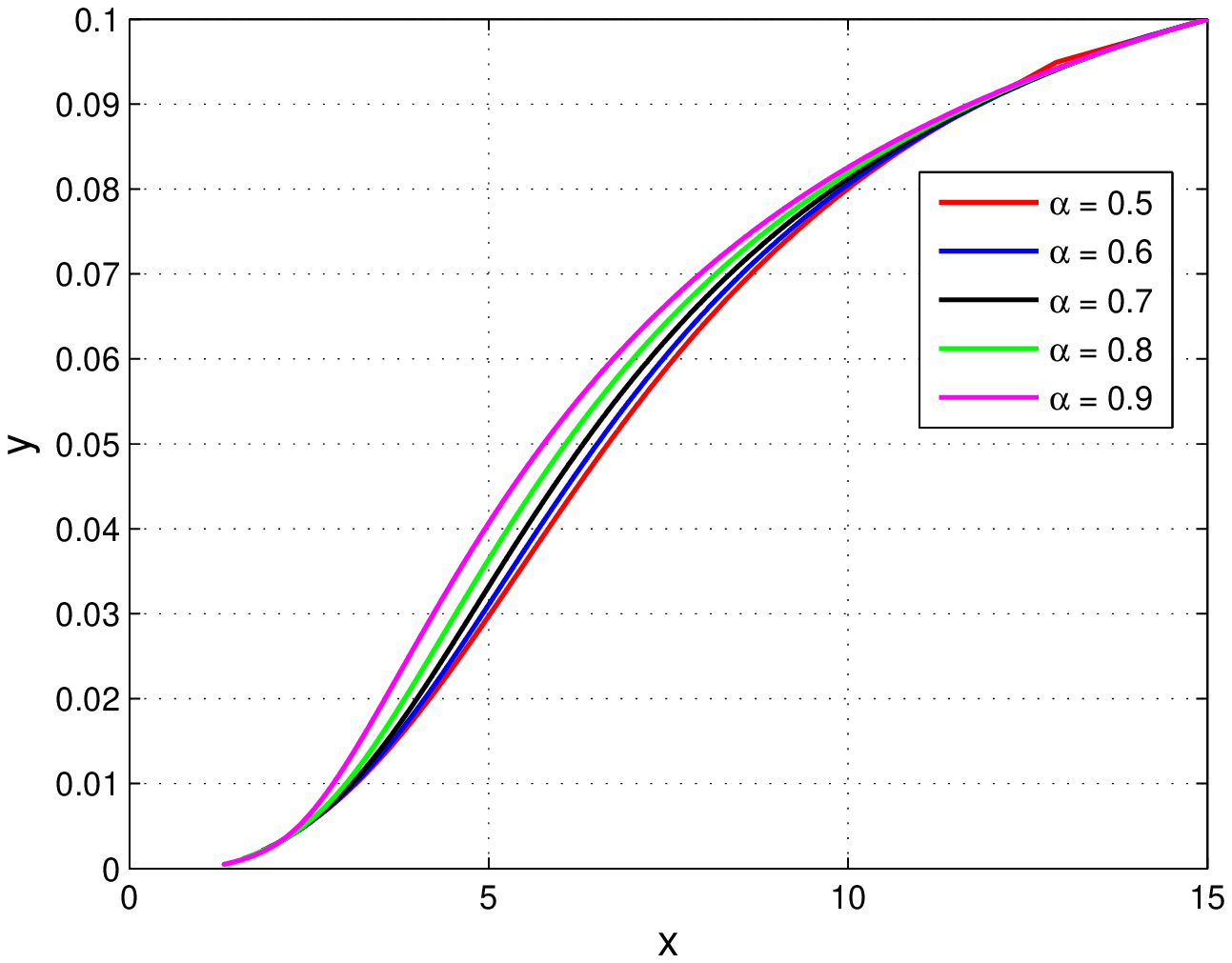} & \includegraphics[width=0.48\textwidth]{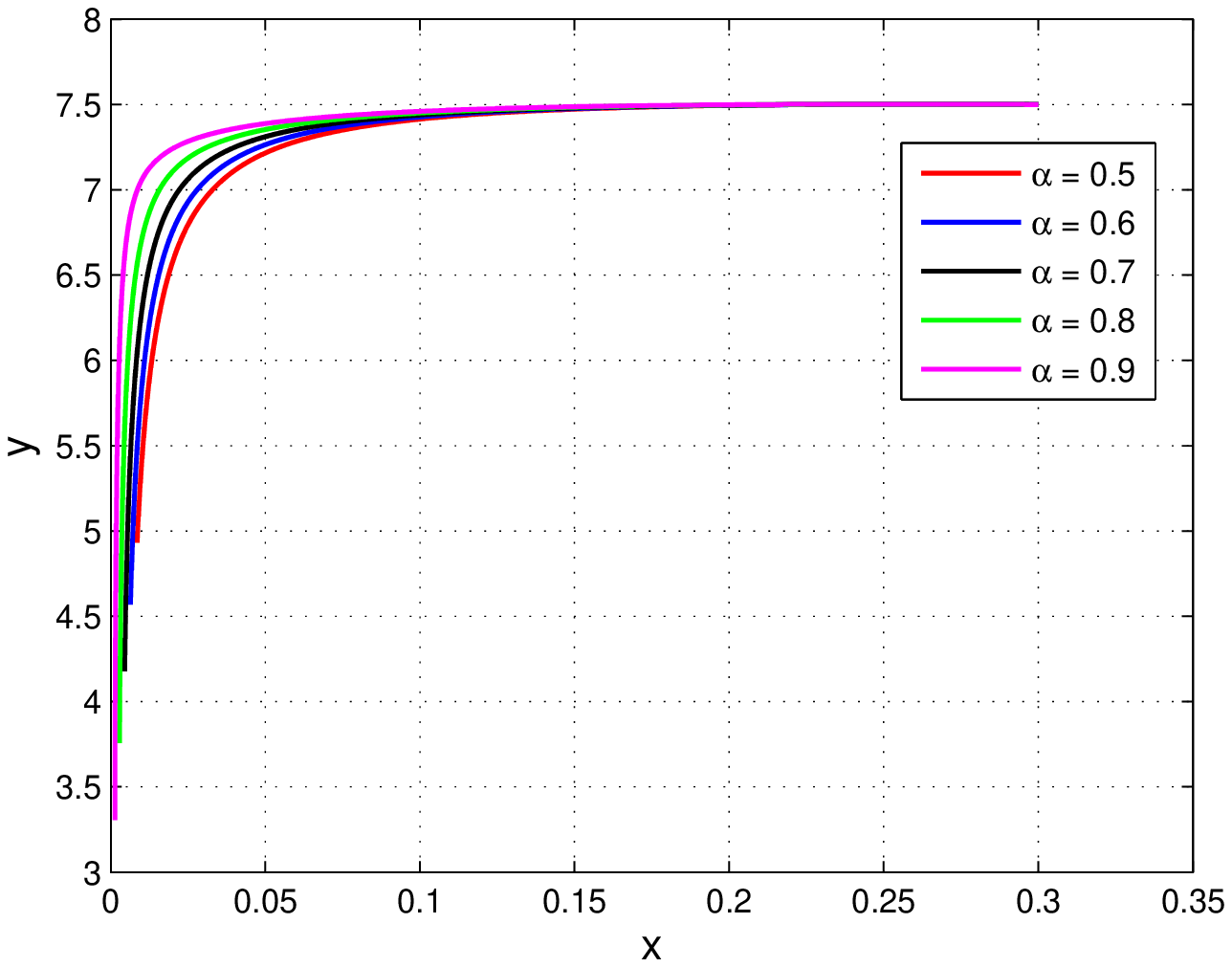}\\
        \footnotesize{(a) Model 1: $x_0 =15$, $y_0 =0.1$} & \footnotesize{(b) Model 2: $x_0 =0.3$, $y_0 =7.5$}\\
    \end{tabular}}
    \caption{Dynamical behavior of solutions obtained using a very small time step for discretization: $h=0.0001$.}
\end{figure}
\end{center}
It can be also interesting to test the numerical method for a fixed $\alpha$, i.e. $\alpha =0.8$ in the following directly looking for a solution of the fractional system, results can be seen on Figure 2.

\begin{center}
\begin{figure}[ht]
\centerline{%
    \begin{tabular}{cc}
        \includegraphics[width=0.48\textwidth]{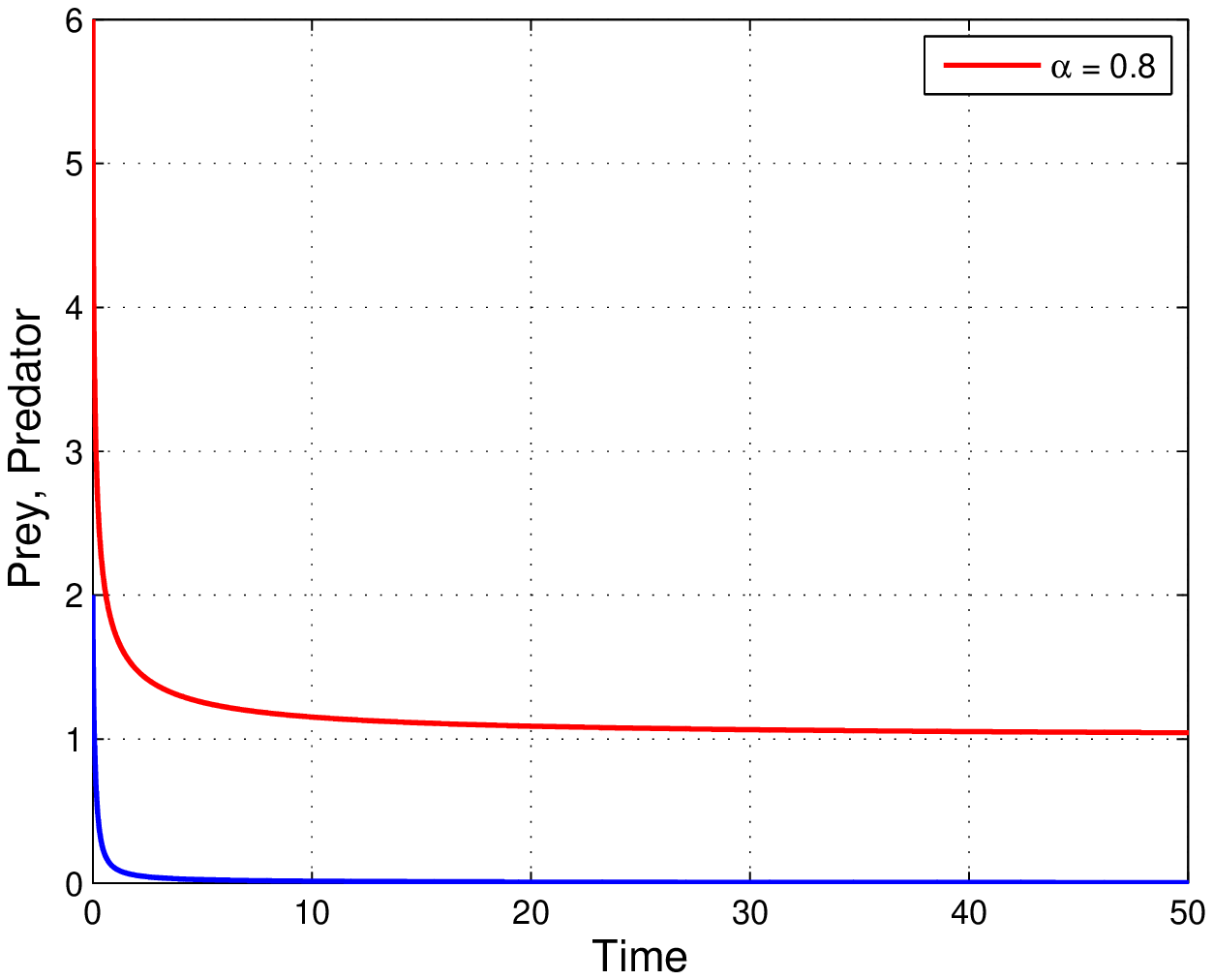} & \includegraphics[width=0.48\textwidth]{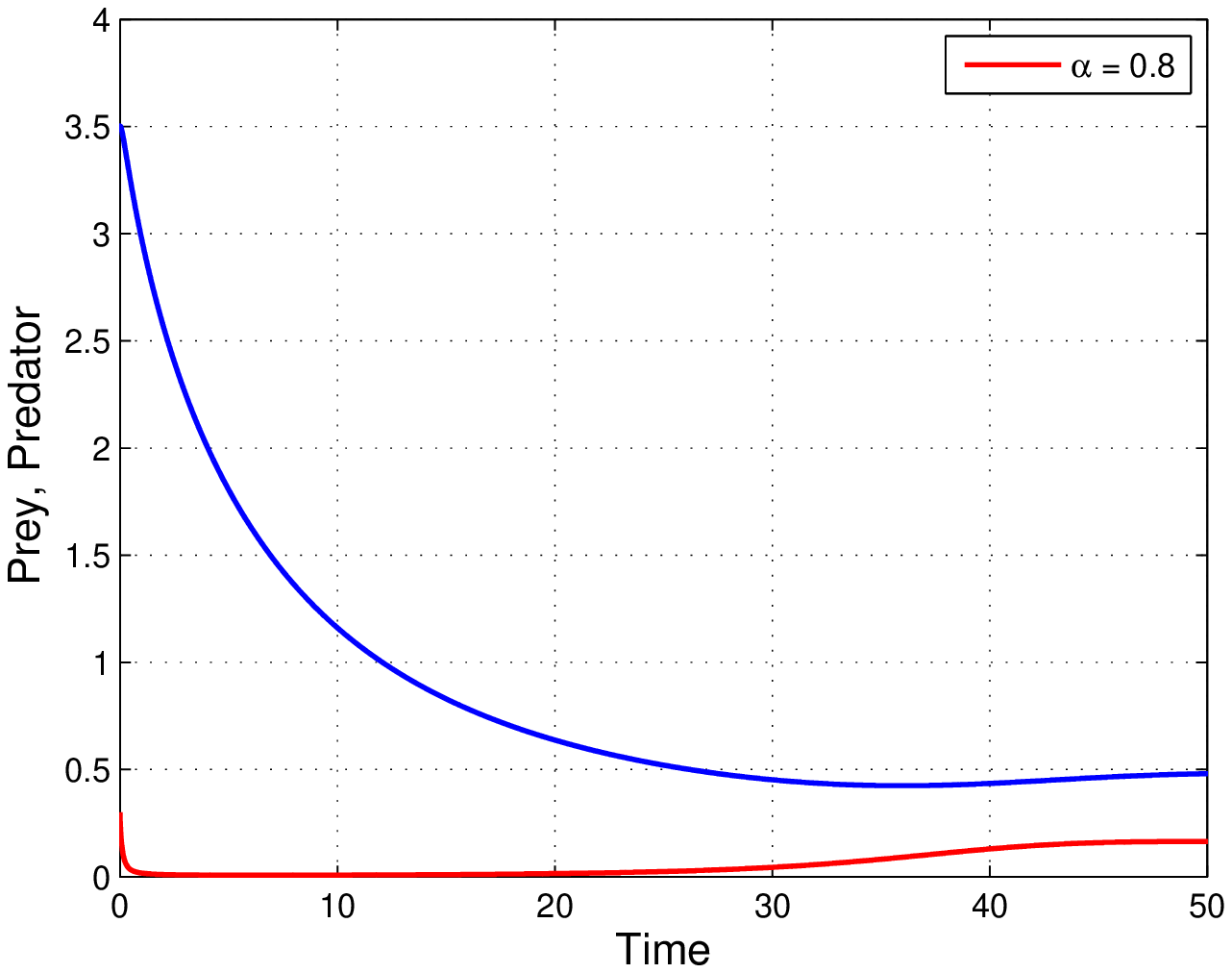}\\
        \footnotesize{(a) Model 1: $x_0 =6$, $y_0 =2$} & \footnotesize{(b) Model 2: $x_0 =0.3$, $y_0 =3.5$}\\
    \end{tabular}}
    \caption{Prey and Predator densities obtained with $\alpha=0.8$ and time step $h=0.001$.}
\end{figure}
\end{center}

It is important to observe that the solutions are remaining positive all over the time.

\subsection{Lost of positivity of solutions}

We take under consideration the Gr\"unwald-Letnikov difference method for model 1 and 2 with appropriate initial conditions. Results can be seen on Figure 3, and we observe that un-physical negative values of densities are produced.

\begin{center}
\begin{figure}[ht]
\centerline{%
    \begin{tabular}{cc}
        \includegraphics[width=0.48\textwidth]{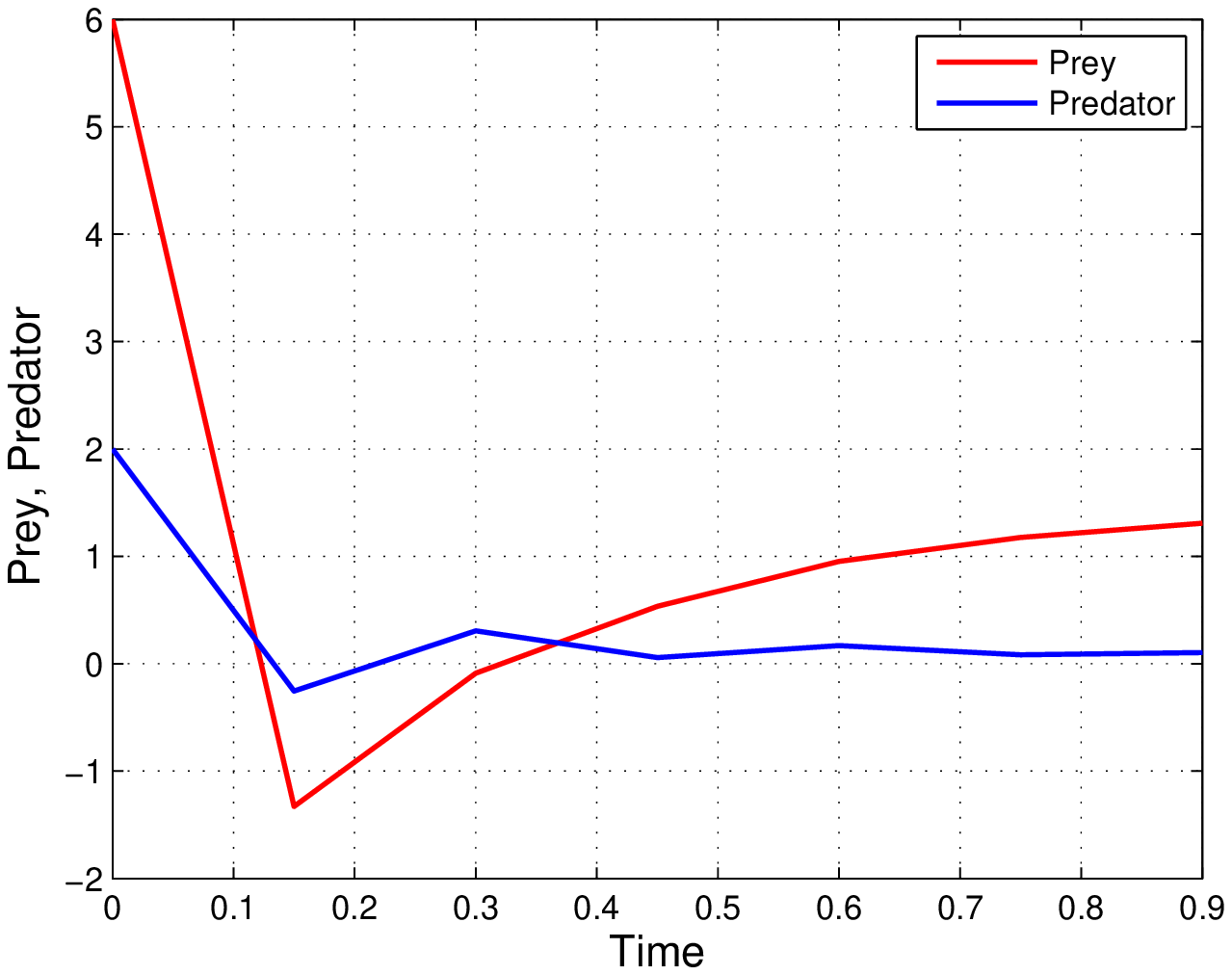} & \includegraphics[width=0.48\textwidth]{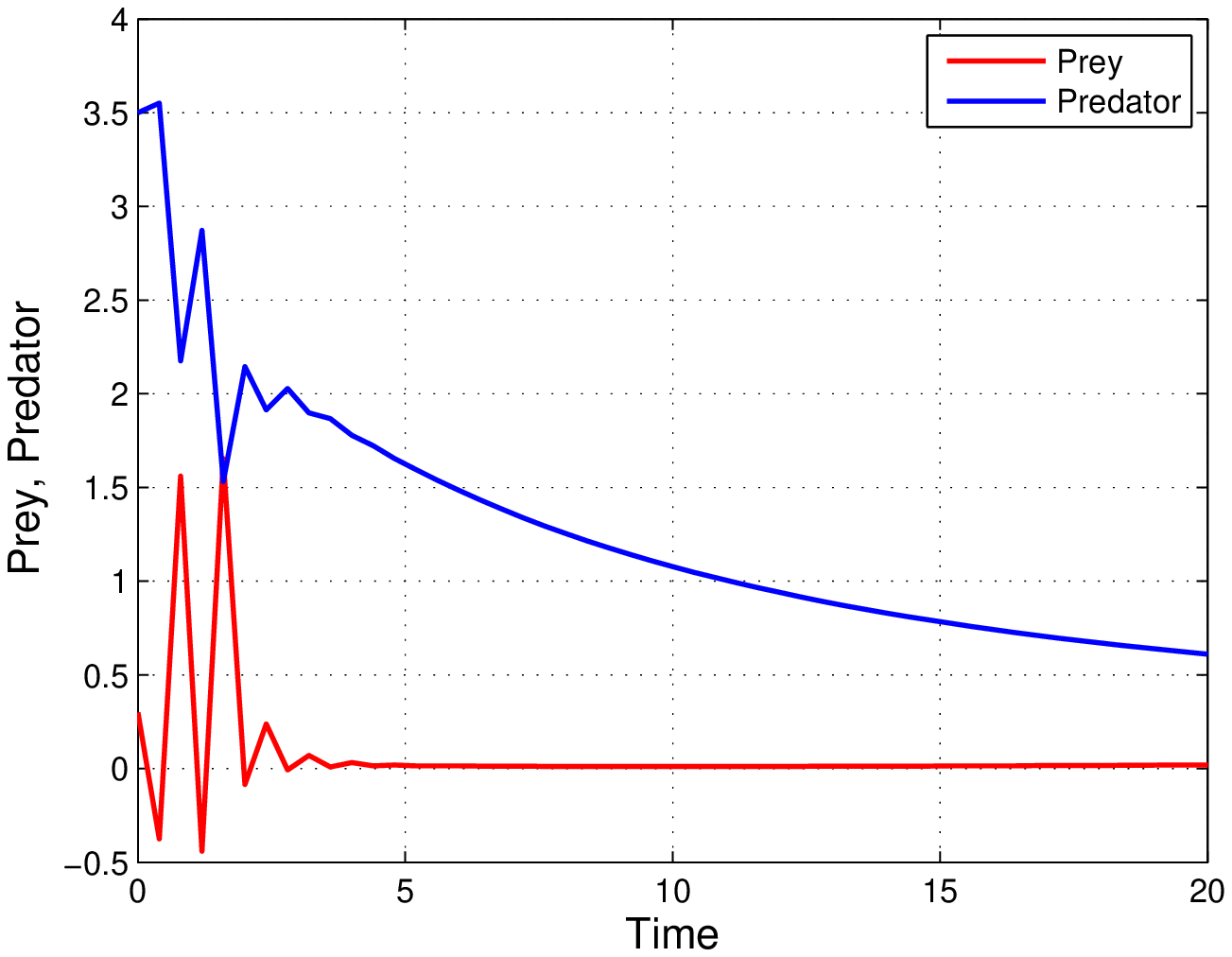}\\
        \footnotesize{(a) Model 1: $x_0 =6$, $y_0 =2$, $h=0.15$} & \footnotesize{(b) Model 2: $x_0 =0.3$, $y_0 =3.5$, $h=0.4$}\\
    \end{tabular}}
    \caption{Results obtained for Gr\"unwald-Letnikov method with $\alpha=0.8$.}
\end{figure}
\end{center}

\subsection{Lost of stability for equilibrium points}

We analyze numerically model 1 and 2 using Gr\"unwald-Letnikov difference method. As can be seen on the Figure 1, we are waiting for a convergence to the equilibrium point. However, what can be observed on Figure 4, the Gr\"unwald-Letnikov method induces a break of stability for model 1. The analysis of model 2 shows that even for small value of the time increment we obtain a large oscillations which are not existing.

\begin{center}
\begin{figure}[ht]
\centerline{%
    \begin{tabular}{cc}
        \includegraphics[width=0.48\textwidth]{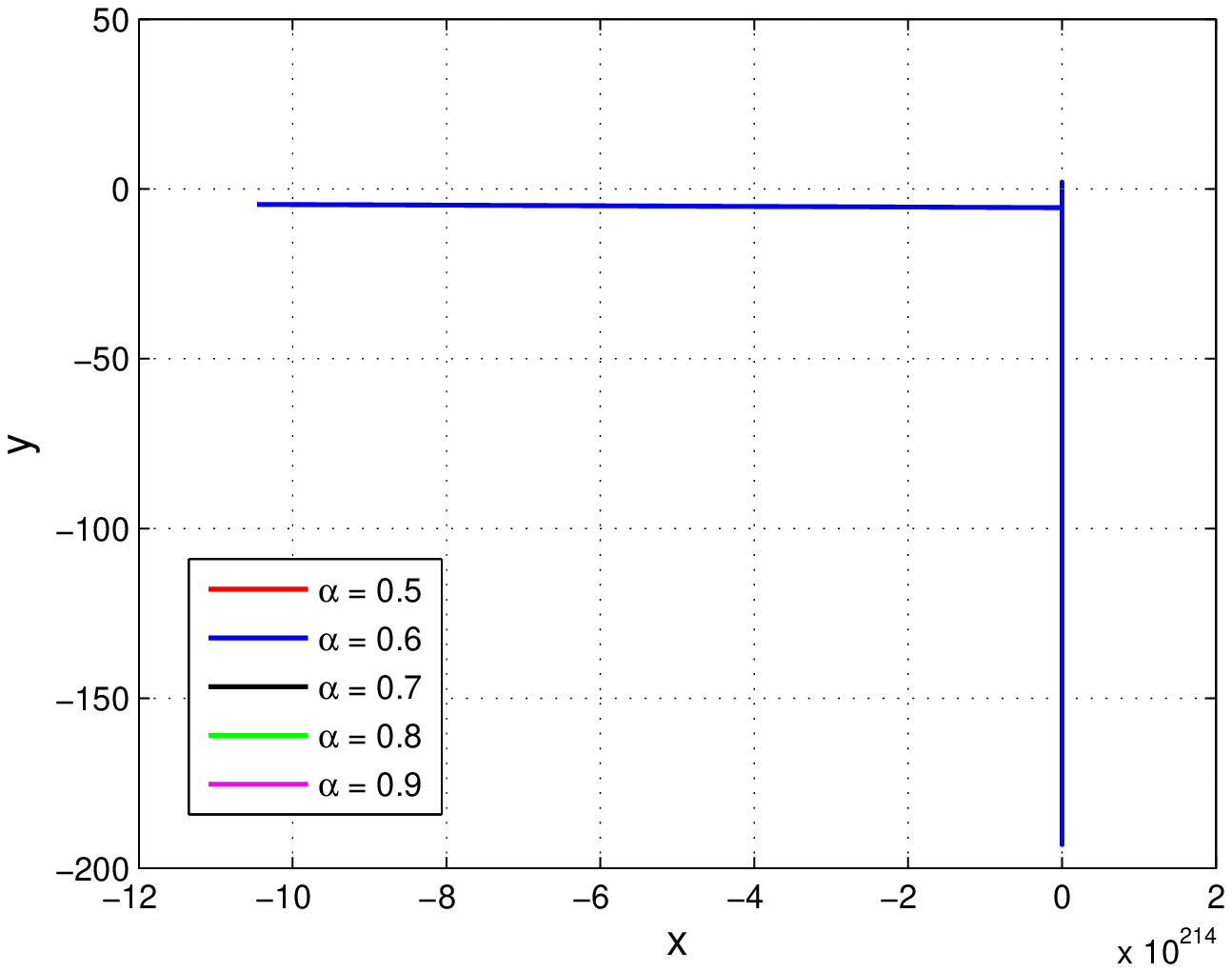} & \includegraphics[width=0.48\textwidth]{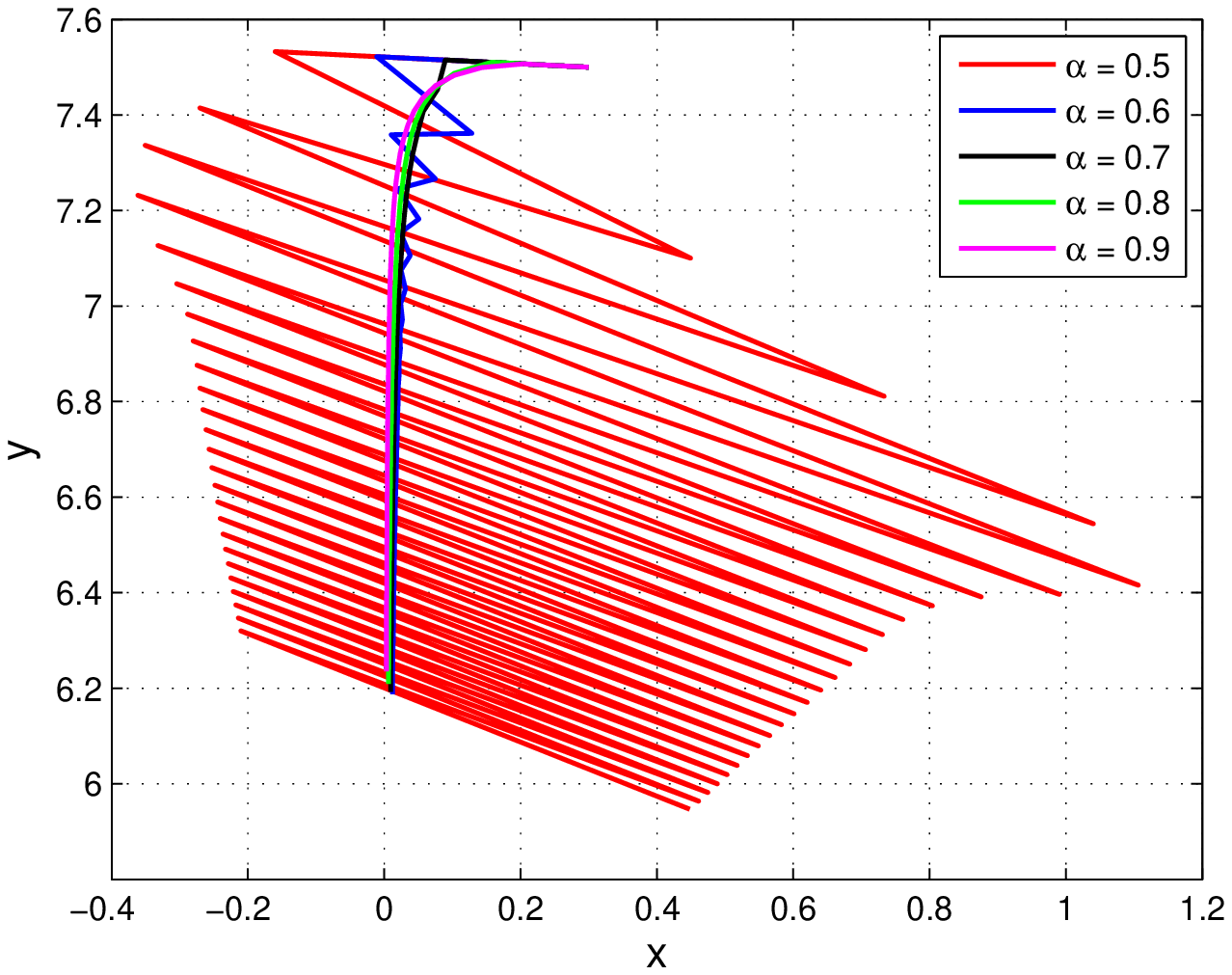}\\
        \footnotesize{(a) Model 1: $x_0 =15$, $y_0 =0.1$, $h=0.03$} & \footnotesize{(b) Model 2: $x_0 =0.3$, $y_0 =7.5$, $h=0.02$}\\
    \end{tabular}}
    \caption{Results obtained for Gr\"unwald-Letnikov method for different values of $\alpha$.}
\end{figure}
\end{center}

\section{Definition of the NSFD for fractional systems}

For discretization of (\ref{frac}) we assume that the integration is over the interval $[0,T]$ and we consider the equidistance partition of this interval with $t_i=t_0+ih$, $i=1,\ldots,N=\frac{T}{h}$, as nodal points with the time step $h>0$.\\

We discretize the Caputo derivative by the Gr{\"u}nwald-Letnikov difference operator.

To ensure the positivity of the NSFD scheme we split the right hand sides of equations of the model on the positive and negative part. The positive part we discretize using the previous time level and for the negative part we use the nonlocal discretization, to provide linearity with respect to the actual time level.

\begin{df}[NSFD for fractional systems]
\label{nsfd}
The Non Standard finite difference scheme for fractional systems of the form (\ref{frac}) is given by
\begin{equation}
\label{met}
{}_c \Delta_{a+} [X]_n  = f_+ (x_{n-1} ) -x_{n} f_- (x_{n-1} ), n=0,\dots ,N-1.
\end{equation}
\end{df}

We have tried to produce an abstract definition of a NSFD scheme for a general fractional systems in the spirit of \cite{ang1,ang2} but it can be easily done replacing the classical assumption on the approximation of the derivative by an adapted one on the fractional derivative.

\section{Positivity of the fractional NSFD scheme}

Our class of systems contains classical models for population dynamics dealing with densities which must remain positive. In order to be able to use numerical simulations to give insights on the model, we need to be sure that the positivity of each variable is preserved in order to avoid un-physical data. Our first theorem give a positive answer for our scheme :

\begin{twr}
\label{thw1}
The fractional NSFD scheme (\ref{met}) preserves positivity.
\end{twr}

\begin{proof}
According to (\ref{met}) we obtain
\begin{equation}
x_{n} \left [ 1+h^{\alpha} f_- (x_{n-1} ) \right ] = h^{\alpha} f_+ (x_{n-1} ) +x_0 \di\sum_{k=0}^{n-1} \alpha_k^{[\alpha]} - \di\sum_{k=1}^{n-1} \alpha_k^{[\alpha]} x_{n-k} ,
\end{equation}
which gives, reindexing the last term
\begin{equation}
x_{n} \left [ 1+h^{\alpha} f_- (x_{n-1} ) \right ] = h^{\alpha} f_+ (x_{n-1} ) +x_0 \di\sum_{k=0}^{n-1} \alpha_k^{[\alpha]} - \di\sum_{k=0}^{n-2} \alpha_{k+1}^{[\alpha]} x_{n-1-k} .
\end{equation}
As $\di\sum_{k=0}^{n-1} \alpha_k^{[\alpha]} >0$, we only have to prove that the last term is negative when the $x_j \geq 0$ for $j\leq n-1$ in order to obtain positivity of the scheme for all $h>0$.

The main observation is that for $k\geq 1$, we have $\alpha_k^{[\alpha]} \leq 0$ as long as $0<\alpha \leq 1$. As for $j\leq n-1$, we assume that all $x_j \geq 0$, we deduce that
$\di\sum_{k=0}^{n-2} \alpha_{k+1}^{[\alpha]} x_{n-1-k} \leq 0$. This concludes the proof.
\end{proof}

\section{Preservation of the set of equilibrium points}

A classical problem in numerical analysis is that some methods produce "ghost" or "fake" equilibrium points. We refer to \cite{cp} for some examples dealing with the Runge-Kutta (RK) method of order $2$. If we denote by $\mathcal{E}$ the set of equilibrium points of the differential equation and by $\mathcal{F}_h$ the set of fixed point for the numerical scheme, we have in general $\mathcal{E} \subset \mathcal{F}$. One easily prove that the set of equilibrium is {\it exactly} preserved by the GL scheme, meaning that we have $\mathcal{F}=\mathcal{E}$.

\begin{twr}
The GL scheme preserves exactly the set of equilibrium points.
\end{twr}

\begin{proof}
If $x_0$ is a fixed point of our numerical scheme, i.e. $x_n =x_0$ for $n\in \N$, then the Grunwald-Letnikov derivative is zero. Replacing in the GL scheme, we then obtain
\begin{equation}
f (x_0) =0,
\end{equation}
which corresponds exactly to the characterization of equilibrium points for the differential equation.
\end{proof}

The same argument can be used to prove that the NSFD scheme preserves also exactly the set of equilibrium points. Indeed, the Grunwald-Letnikov derivative will be also zero, and the shift in the right-end side does not make any difference as $x_n =x_0$ for all $n\in \N$. We then have :

\begin{twr}
The NSFD scheme preserves exactly the set of equilibrium points.
\end{twr}

\section{Convergence of the fractional NSFD scheme}

Let us denote by $x(t)$ an exact solution of the fractional differential systems (\ref{frac}) with initial condition $x_0 \in \mathbb{R}^m$. The associated numerical solution is denoted by $x_n$, $n\geq 0$ and defined by Definition \ref{nsfd}. We denote by $\phi_n (x_0 ,\dots ,x_{n-1} )$ the application defined by equation (\ref{met}).\\

We introduce the following error terms for all $n\geq 0$ :

\begin{equation} \label{err}
\left .
\begin{array}{lll}
e_n & = & x_n -x (t_n ) , \\
\tau_n  & = & \phi_n (x(t_0), x(t_1) ,\dots ,x(t_{n-1} )) - x (t_n ),\\
\mu_{n,l} & = & \phi_n (x_0, \dots ,x _{n-l} , x(t_{n-l+1} ), \dots , x (t_{n-1} ) )\\
 & & - \phi_n (x_0, \dots ,x(t_{n-l}),x (t_{n-l+1} ) , \dots , x (t_{n-1} ) ) ,\\
 & &  l=1,\dots ,n.
\end{array}
\right .
\end{equation}
called the {\it global error}, {\it local truncation error} and {\it computational error} respectively. By assumption, the application $\phi_n$ is defined by
\begin{equation}
\phi_n (z_0 ,\dots ,z_{n-1})= \frac{1}{\left [ 1+h^{\alpha} f_- (z_{n-1}) \right ]}
\left (
h^{\alpha} f_+ (z_{n-1}) +z_0 \di\sum_{k=0}^{n-1} \alpha_k^{[\alpha]} - \di\sum_{k=1}^{n-1} \alpha_k^{[\alpha]} z_{n-k}
\right )
.
\end{equation}

We have for all $n\geq 0$ :
\begin{equation}
e_n =\phi_n(x_0,\ldots,x_{n-1})-\phi_n(x(t_0),\ldots,x(t_{n-1}))+\tau_n=\sum_{l=1}^{n} \mu_{n,l} +\tau_n .
\end{equation}

This Section is devoted to the proof of the following convergence result :

\begin{twr}
\label{thw:conv}
Let $f_{-}$ and $f_{+}$ are lipschitz functions with constants $L_{-}$ and $L_{+}$ respectively and the analytical solution $x\in C^{1}([0,T],\R^m)$, then the numerical scheme is convergent with order $\mathcal{O}(h^\alpha)$.
\end{twr}

In fact we present numerical evidence that the order of convergence is often better than $\alpha$. We consider the model (\ref{u}) which is discussed in the next section. Since we do not know the analytical solution of the model, we take as an exact solution the approximation computed with a sufficiently small time step $h^{\star}$. In order to compare we take $h^{\star}=2^{-12}$. Let $x^h$ defines the approximated solution computed with a step $h$ and by $x^{h^{\star}}$ we denote the approximation obtained with $h^{\star}$ on the interval $[0,T]$. We compare the solution $x^h$ with $x^{h^{\star}}$ on a grid corresponding to the larger step $h$. The same nomenclature we use for the numerical solution $y$. We define the maximum errors
$$
\begin{array}{rl}
\varepsilon_x(h)&=\max{\{|x_n^h-x_n^{h^{\star}}|:n=0,\ldots,N\}}\\
\varepsilon_y(h)&=\max{\{|y_n^h-y_n^{h^{\star}}|:n=0,\ldots,N\}}
\end{array}
$$
$$
\xi(h)=\max{\{\varepsilon_x(h),\varepsilon_y(h)\}}
$$
and the standard rate of convergence
$$
\rho_{\alpha}=\log_2{\Big(\frac{\varepsilon(2^{\alpha}h^{\alpha})}{\varepsilon(h^{\alpha})}\Big)}.
$$
Table 1 presents the errors of the approximations computed by the applications of NFDS scheme with different step sizes and the order of convergence with respect to the different values of $\alpha$.
\begin{table}[ht]
\begin{tabularx}{\linewidth}{p{0.7cm}p{0.7cm}XXXXX}\hline
\multirow{2}{0.6cm}{$\alpha$} & & \multicolumn{4}{c}{$h$}\\ \cline{3-7}
 & & $2^{-3}$ & $2^{-4}$ & $2^{-5}$ & $2^{-6}$ & $2^{-7}$\\    \hline
\multirow{2}{0.6cm}{$0.5$} & $\xi$ & 2.093e-3 & 1.462e-3 & 1.007e-3 & 6.884e-4 & 4.663e-4\\
 & $\rho_{\alpha}$ & & 0.5170 & 0.5374 & 0.5499 & 0.5619 \\    \hline
\multirow{2}{0.6cm}{$0.6$} & $\xi$ & 1.758e-3 & 1.104e-3 & 6.883e-4 & 4.289e-4 & 2.672e-4\\
 & $\rho_{\alpha}$ & & 0.6709 & 0.6819 & 0.6821 & 0.6829 \\    \hline
 \multirow{2}{0.6cm}{$0.7$} & $\xi$ & 1.366e-3 & 7.618e-4 & 4.269e-4 & 2.423e-4 & 1.389e-4\\
 & $\rho_{\alpha}$ & & 0.8426 & 0.8355 & 0.8168 & 0.8025 \\    \hline
 \multirow{2}{0.6cm}{$0.8$} & $\xi$ & 1.094e-3 & 5.751e-4 & 2.965e-4 & 1.492e-4 & 7.389e-5\\
 & $\rho_{\alpha}$ & & 0.9284 & 0.9558 & 0.9899 & 1.0147 \\    \hline
 \multirow{2}{0.6cm}{$0.9$} & $\xi$ & 8.261e-4 & 4.335e-4 & 2.315e-4 & 1.185e-4 & 5.868e-5\\
 & $\rho_{\alpha}$ & & 0.9300 & 0.9050 & 0.9657 & 1.0142 \\    \hline
\end{tabularx}
\caption{Initial condition $(0.05,0.05)$, time interval $[0,1]$.}
\end{table}
\begin{center}
\begin{figure}[ht]
\centerline{%
    \begin{tabular}{cc}
        \includegraphics[width=0.5\textwidth]{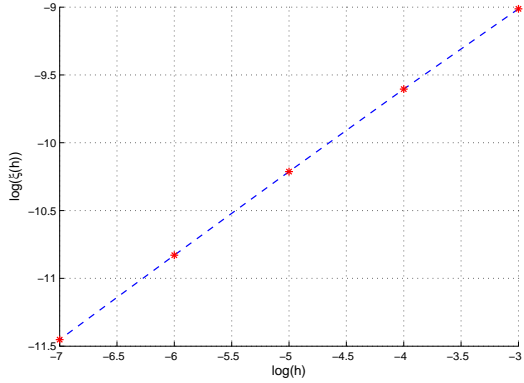} & \includegraphics[width=0.5\textwidth]{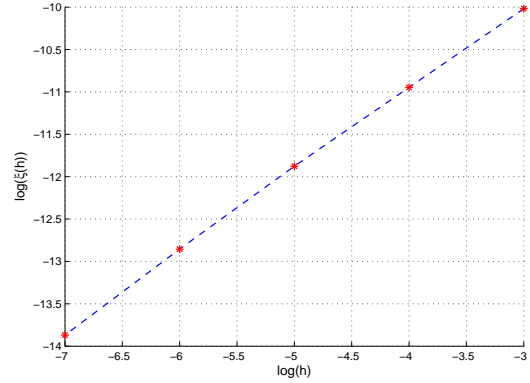}\\
        (a) $\alpha=0.5$ & (b) $\alpha=0.8$\\
    \end{tabular}}
    \caption{Order of convergence in a $log$ scale for initial condition $(0.05,0.05)$ and time interval $[0,1]$.}
\end{figure}
\end{center}
The proof of the Theorem \ref{thw:conv} follows the standard procedure : we estimate the global error of the numerical scheme looking on the local truncation error and the computational one.

\subsection{Local truncation error}

\begin{lem}
\label{lem:cons}
{\em (Local truncation error)}
Suppose that the assumptions of Theorem \ref{thw:conv} are satisfied, then the local truncation error of the fractional difference scheme (\ref{met}) is of order $\mathcal{O}(h^{1+\alpha})$.
\end{lem}
\begin{proof}
It is more convenient to write numerical method (\ref{met}) in the form
\begin{equation}
x_n+\sum_{k=0}^n \alpha_{k}^{[\alpha]}x_{n-k}-\alpha_n^{[\alpha-1]}x_0=h^{\alpha}f_{+}(x_{n-1})-h^{\alpha} x_n f_{-}(x_{n-1}).
\end{equation}
Once we add and subtract appropriate terms we obtain
\begin{equation}
\begin{split}
x_n-&x(t_n)+x(t_n)+\sum_{k=0}^n \alpha_{k}^{[\alpha]}x_{n-k}-\alpha_n^{[\alpha-1]}x_0=h^{\alpha}f_{+}(x_{n-1})-h^{\alpha} x_n f_{-}(x_{n-1})\\
-&h^{\alpha}f_{+}(x(t_{n}))+h^{\alpha}f_{+}(x(t_{n}))-h^{\alpha} x(t_n) f_{-}(x(t_n))+h^{\alpha} x(t_n) f_{-}(x(t_{n})).
\end{split}
\end{equation}
We examine the error of the method on the step $n$ therefore we assume that there is no local errors for $i=0,1,\ldots,n-1$, i.e. $x_j=x(t_j)$. It follows
\begin{equation}
\begin{split}
x_n-&x(t_n)+x(t_n)+\sum_{k=0}^n \alpha_{k}^{[\alpha]}x(t_{n-k})-\alpha_n^{[\alpha-1]}x(t_0)=h^{\alpha}f_{+}(x(t_{n-1}))-h^{\alpha} x_n f_{-}(x(t_{n-1}))\\
-&h^{\alpha}f_{+}(x(t_{n}))+h^{\alpha}f_{+}(x(t_{n}))-h^{\alpha} x(t_n) f_{-}(x(t_{n}))+h^{\alpha} x(t_n) f_{-}(x(t_{n})).
\end{split}
\end{equation}
Based on (\ref{est}) we have
\begin{equation}
\begin{split}
x_n-&x(t_n)+h^{\alpha}\, _{c}D_{0,t}^{\alpha}x(t_n)=h^{\alpha}f_{+}(x(t_{n-1}))-h^{\alpha} x_n f_{-}(x(t_{n-1}))\\
-&h^{\alpha}f_{+}(x(t_{n}))+h^{\alpha}f_{+}(x(t_{n}))-h^{\alpha} x(t_n) f_{-}(x(t_{n}))+h^{\alpha} x(t_n) f_{-}(x(t_{n}))+\mathcal{O}(h^{1+\alpha}).
\end{split}
\end{equation}
We eliminate identical parts and group appropriate terms together and we obtain
\begin{equation}
\begin{array}{rcl}
x_n-x(t_n)&= & h^{\alpha}\Big(f_{+}(x(t_{n-1}))-f_{+}(x(t_{n}))\Big)\\
 &+& h^{\alpha}\Big(x(t_n) f_{-}(x(t_{n}))-x_n f_{-}(x(t_{n-1}))\Big)+\mathcal{O}(h^{1+\alpha}).\end{array}
\end{equation}
According to definition of the local truncation error $\tau_n$ (\ref{err}) we deduce
\begin{equation}
\begin{split}
\tau_n=&\phi_n(x(t_0),\ldots,x(t_{n-1}))-x(t_n)=-x(t_n)\\
+&\frac{1}{\left[1+h^{\alpha}f_-(x(t_{n-1}))\right]}\left(h^{\alpha}f_+(x(t_{n-1}))+x(t_0)\sum_{k=0}^{n-1}\alpha_k^{[\alpha]}- \sum_{k=1}^{n-1}\alpha_k^{[\alpha]}x(t_{n-k})\right)\\
=&\frac{1}{\left[1+h^{\alpha}f_-(x(t_{n-1}))\right]}\left(h^{\alpha}f_+(x(t_{n-1}))+x(t_0)\sum_{k=0}^{n-1}\alpha_k^{[\alpha]}- \sum_{k=1}^{n-1}\alpha_k^{[\alpha]}x(t_{n-k})\right)\\
-&x(t_n)+x_n-\frac{1}{\left[1+h^{\alpha}f_-(x_{n-1})\right]}\left(h^{\alpha}f_+(x_{n-1})+x_0\sum_{k=0}^{n-1}\alpha_k^{[\alpha]}- \sum_{k=1}^{n-1}\alpha_k^{[\alpha]}x_{n-k}\right).
\end{split}
\end{equation}
It follows from assumption $x_j=x(t_j)$, for $j=0,\ldots,n-1$, that
\begin{equation} \label{tau}
\tau_n=x_n-x(t_n).
\end{equation}
Adding and subtracting components to above and taking (\ref{tau}) into account, we have
\begin{equation}
\begin{array}{rcl}
\tau_n &=&\frac{1}{\Big(1+h^{\alpha}f_-(z(t_{n-1}))\Big)}\bigg[h^{\alpha}\Big(f_{+}(z(t_{n-1}))-f_{+}(z(t_{n}))\Big)\\
 &+&h^{\alpha} z(t_n)\Big(f_-(z(t_n))-f_-(z(t_{n-1}))\Big)+\mathcal{O}(h^{1+\alpha})\bigg].\end{array}
\end{equation}
Let the constants $L, C_z, D >0$ be defined as follows
\begin{equation}
\label{const}
L=\max{\{L_{-},L_+\}}, \;\;\; C_z=\max{\{\|z(t)\|:t\in[0,T]\}},\;\;\; D=\max{\{\|z'(t)\|:t\in[0,T]\}}.
\end{equation}
As $1+h^{\alpha}f_-(x)\geq 1$ for all $x\geq 0$ then $\frac{1}{1+h^{\alpha}f_-(x)}\leq 1$ and from assumptions of lemma we obtain
$$
\|\tau_n\|\leq h^{\alpha} L h D +h^{\alpha} C_z L h D + \mathcal{O}(h^{1+\alpha})=h^{1+\alpha} \tilde{C}+\mathcal{O}(h^{1+\alpha}).
$$
Therefore, for sufficiently small step $h$ we conclude the thesis of the lemma.
\end{proof}

\subsection{Computational error}

We have $\mu_{n,l} = \phi_n (x_0, \dots x_{n-l},x(t_{n-l+1}),\ldots ,x(t_{n-1}) ) - \phi_n (x_0, \dots ,x(t_{n-l}),x (t_{n-l+1} ) , \dots , x (t_{n-1} ) ) $.
For $l=2,\ldots,n-1$ we have
\begin{equation}
\begin{split}
\mu_{n,l}=&\frac{1}{\left [ 1+h^{\alpha} f_- (x(t_{n-1})) \right ]}
\left (
h^{\alpha} f_+ (x(t_{n-1})) +x_0 \di\sum_{k=0}^{n-1} \alpha_k^{[\alpha]} - \di\sum_{k=1}^{l} \alpha_k^{[\alpha]} x_{n-k}-\di\sum_{k=l+1}^{n-1} \alpha_k^{[\alpha]} x(t_{n-k})
\right )\\
-&\frac{1}{\left [ 1+h^{\alpha} f_- (x(t_{n-1})) \right ]}
\left (
h^{\alpha} f_+ (x(t_{n-1})) +x_0 \di\sum_{k=0}^{n-1} \alpha_k^{[\alpha]} - \di\sum_{k=1}^{l-1} \alpha_k^{[\alpha]} x_{n-k}-\di\sum_{k=l}^{n-1} \alpha_k^{[\alpha]} x(t_{n-k})
\right )\\
=&\frac{1}{\left [ 1+h^{\alpha} f_- (x(t_{n-1})) \right ]}\alpha_l^{[\alpha]}\left( x(t_{n-l})-x_{n-l}\right)
.
\end{split}
\end{equation}
For $l=n$ and $l=1$ we obtain respectively
\begin{equation}
\begin{split}
\mu_{n,n}=&\frac{1}{\left [ 1+h^{\alpha} f_- (x(t_{n-1})) \right ]} \left( x_0-x(t_0)\right)\di\sum_{k=0}^{n-1} \alpha_k^{[\alpha]},\\
\mu_{n,1}=&\frac{h^{\alpha}f_+(x_{n-1})-\alpha_1^{[\alpha]}x_{n-1}}{\left [ 1+h^{\alpha} f_- (x_{n-1}) \right ]} -\frac{h^{\alpha}f_+(x(t_{n-1}))-\alpha_1^{[\alpha]}x(t_{n-1})}{\left [ 1+h^{\alpha} f_- (x(t_{n-1})) \right ]}.
\end{split}
\end{equation}
According to this, the sum of computational errors equals
\begin{equation}
\begin{split}
\sum_{l=0}^{n-1}\mu_{n,l}=&\frac{1}{\left [ 1+h^{\alpha} f_- (x(t_{n-1})) \right ]}\left (
\left( x_0-x(t_0)\right)\di\sum_{k=0}^{n-1} \alpha_k^{[\alpha]}+\sum_{l=2}^{n-1}\alpha_l^{[\alpha]} (x(t_{n-l})-x_{n-l})
\right )\\
+&\frac{h^{\alpha}(f_+(x_{n-1})-f_+(x(t_{n-1})))}{\left [ 1+h^{\alpha} f_- (x_{n-1}) \right ]} +\frac{h^{2\alpha}f_+(x(t_{n-1}))\Big(f_-(x(t_{n-1}))-f_-(x_{n-1})\Big)}{\left [ 1+h^{\alpha} f_- (x_{n-1}) \right ]\left [ 1+h^{\alpha} f_- (x(t_{n-1})) \right ]}\\
-&\frac{\alpha_1^{[\alpha]}}{\left [ 1+h^{\alpha} f_- (x_{n-1}) \right ]\left [ 1+h^{\alpha} f_- (x(t_{n-1})) \right ]}\Big[x_{n-1}-x(t_{n-1})\\
+&h^{\alpha}(x_{n-1}-x(t_{n-1}))f_-(x(t_{n-1}))+h^{\alpha} x(t_{n-1})(f_-(x(t_{n-1}))-f_-(x_{n-1}))\Big]
\end{split}
\end{equation}
In pursuance to the definition of $e_n$ (\ref{err}) we have
\begin{equation} \label{comp}
\begin{split}
|\sum_{l=0}^{n-1}\mu_{n,l}|\leq &|e_0|\di\sum_{k=0}^{n-1} \alpha_k^{[\alpha]}+\sum_{l=1}^{n-1}(-\alpha_l^{[\alpha]}) |e_{n-l}|
+h^{\alpha}|f_+(x_{n-1})-f_+(x(t_{n-1}))| \\
+&h^{2\alpha}|f_+(x(t_{n-1}))|\Big|f_-(x(t_{n-1}))-f_-(x_{n-1})\Big|+h^{\alpha}|e_{n-1}||f_-(x(t_{n-1}))|\\
+&h^{\alpha}|x(t_{n-1})|\Big|f_-(x(t_{n-1}))-f_-(x_{n-1})\Big|\\
\leq & |e_0|\di\sum_{k=0}^{n-1} \alpha_k^{[\alpha]}+\sum_{l=1}^{n-1}(-\alpha_l^{[\alpha]}) |e_{n-l}|+h^{\alpha}(L+M+C_zL)|e_{n-1}|+h^{2\alpha}ML|e_{n-1}|,
\end{split}
\end{equation}
where constant $M$ is defined as follows : $|f_{\pm,i}(x)|\leq M$ for all $i=1,\ldots,M$ and constants $L$ and $C_z$ are defined as in (\ref{const}).

\subsection{Proof of Theorem \ref{thw:conv}}

Once we already receive the information about the order of the local truncation error and estimated the computational error, we are able to obtain the final result for the convergence of the method (\ref{met}). Define $\xi_n=\max{\{|\varepsilon_{in}|:1\leq i\leq m\}}$. It follows from Lemma \ref{lem:cons} and (\ref{comp}) that $\xi_n$ satisfies the following difference inequality
$$
\xi_n\leq h^{\alpha}(L+M+C_zL)\xi_{n-1}+h^{2\alpha} M L\xi_{n-1}+\sum_{k=1}^n(-\alpha_k^{[\alpha]})\xi_{n-k}+\xi_0\alpha_n^{[\alpha-1]}+\|\tau_n\|
$$
As $(-\alpha_i^{[\alpha]})\leq 1$ for $i=1,\ldots,n$ and $\alpha_n^{[\alpha-1]}\leq 1$ then without loosing of generality we can write
$$
\xi_n\leq h^{\alpha}(L+M+C_zL+h^{\alpha}ML)\xi_{n-1}+\sum_{k=0}^{n-1}\xi_k+\xi_0+\|\tau_n\|.
$$
Summing above inequality with respect to $n$ we get
$$
\xi_j\leq \left[ h^{\alpha}(L+M+C_zL+h^{\alpha}ML)-1\right]\sum_{n=0}^{j-1}\xi_{n}+\sum_{n=0}^{j-1}(j-n)\xi_n+(j+1)\xi_0+\sum_{n=0}^{j-1}\|\tau_n\|
$$
and then
$$
\xi_j\leq \sum_{n=0}^{j-1}\left[h^{\alpha}(L+M+C_zL+h^{\alpha}ML)-1+j-n\right]\xi_n+j\xi_0+\sum_{n=0}^{j-1}\|\tau_n\|+\xi_0.
$$
We denote the constants $w_n$ as follows
$$
\begin{array}{ll}
w_0 =& h^{\alpha}(L+M+C_zL+h^{\alpha}ML)-1+2j,\\
w_n =& h^{\alpha}(L+M+C_zL+h^{\alpha}ML)-1+j-n,\;\;\; n=1,2,\ldots,j-1.
\end{array}
$$
From definition $w_n>0$ for all $0\leq n\leq j-1$. Finally we get the following recurrence inequality
\begin{equation}
\label{inequa}
\xi_j\leq \sum_{n=0}^{j-1}w_n \xi_n+\sum_{n=0}^{j-1}\|\tau_n\|+\xi_0.
\end{equation}

Moreover, we have the following discrete version of the Gronwall Lemma :

\begin{lem}
\label{lem:gron}
{\em (Discrete Gronwall lemma)\cite{QV}}.
Assume that $\{w_n\}_{n\geq 0}$ is a non-negative sequence, and that the sequence $\{\phi_n\}_{n\geq 0}$ satisfies
\begin{equation} \label{np}
\begin{split}
\phi_0\leq & g_0,\\
\phi_n\leq & g_0+\sum_{k=0}^{n-1}p_k+\sum_{k=0}^{n-1}w_k\phi_k,\;\;\; n\geq 1,
\end{split}
\end{equation}
then if $g_0\geq 0$ and $p_k\geq 0$ for all $k\geq 0$, $\phi_n$ satisfies
\begin{equation} \label{thesis}
\phi_n\leq \left(g_0+\sum_{k=0}^{n-1}p_k\right)\exp{\left(\sum_{k=0}^{n-1}w_k\right)},\;\;\; n\geq 1.
\end{equation}
\end{lem}

Applying this Lemma to inequality (\ref{inequa}), we obtain
\begin{equation}
\begin{split}
\xi_j\leq& \left(\xi_0+\sum_{n=0}^{j-1}\|\tau_n\|\right)\exp\bigg[\sum_{n=1}^{j-1}\Big(h^{\alpha}(L+M+C_zL+h^{\alpha}ML)-1+j-n\Big)\\
+&h^{\alpha}(L+M+C_zL+h^{\alpha}ML)-1+2j\bigg] .
\end{split}
\end{equation}
From Lemma \ref{lem:cons} there exists a constant $C>0$ such that $\|\tau_n\|\leq Ch^{1+\alpha}$. Therefore we have
$$
\xi_j\leq \left(\xi_0+TCh^{\alpha}\right)\exp{\left[ TL(L+M+C_zL+TML)+\frac{1}{2}j(j-1) \right]}\leq \mathcal{O}(\xi_0)+\mathcal{O}(h^{\alpha}).
$$
We conclude the convergence of numerical solution $x$ with order of convergence equals $\mathcal{O}(\xi_0)+\mathcal{O}(h^{\alpha})$.

\section{Numerical example and comparison with the GL scheme}

In the comparison to classical GL scheme we introduce the results obtained by application of NSFD method to the same test model (\ref{exmodel}).

\subsection{Positivity of solutions}

We apply NSFD method for the set of parameters fixed in model 1 and model 2.

\begin{center}
\begin{figure}[ht]
\centerline{%
    \begin{tabular}{cc}
        \includegraphics[width=0.48\textwidth]{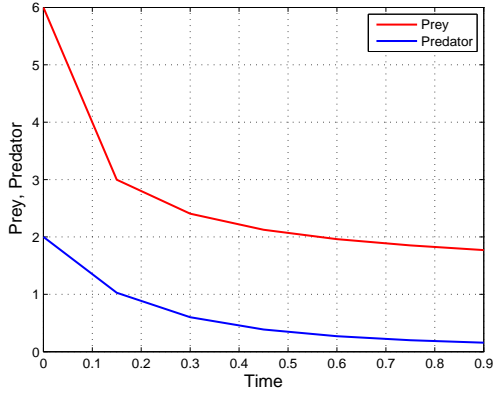} & \includegraphics[width=0.48\textwidth]{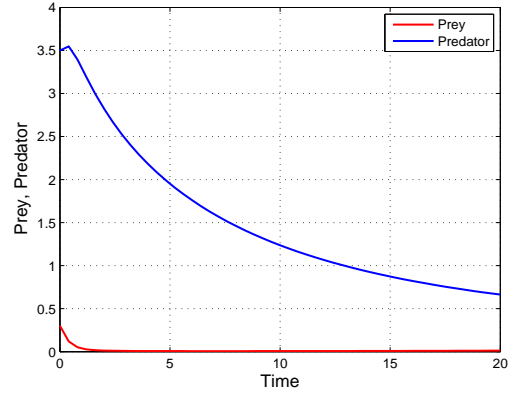}\\
        \footnotesize{(a) Model 1: $x_0 =6$, $y_0 =2$, $h=0.15$} & \footnotesize{(b) Model 2: $x_0 =0.3$, $y_0 =3.5$, $h=0.4$}\\
    \end{tabular}}
    \caption{The NSFD method with $\alpha=0.8$.}
\end{figure}
\end{center}

We observe that unlike the GL scheme (Figure 3) the nonstandard finite difference method preserves positivity even for large step size.

\subsection{Stability for equilibrium points}

Similarly as in the case of positivity, we can observe that the behavior of numerical solutions obtained by the application of NSFD scheme is consistent with behavior of analytical solutions of model (\ref{exmodel}).

\begin{center}
\begin{figure}[h]
\centerline{%
    \begin{tabular}{cc}
        \includegraphics[width=0.48\textwidth]{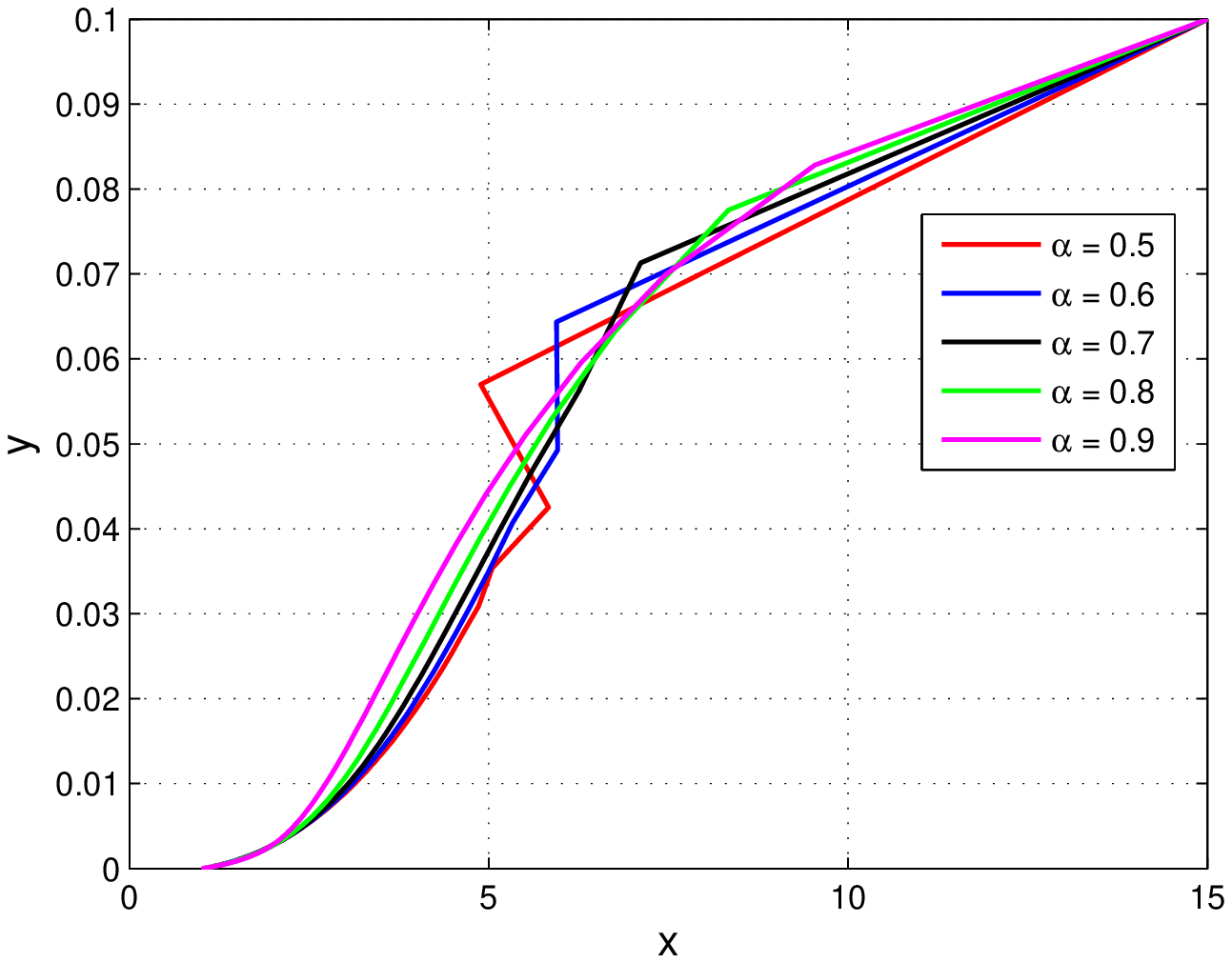} & \includegraphics[width=0.48\textwidth]{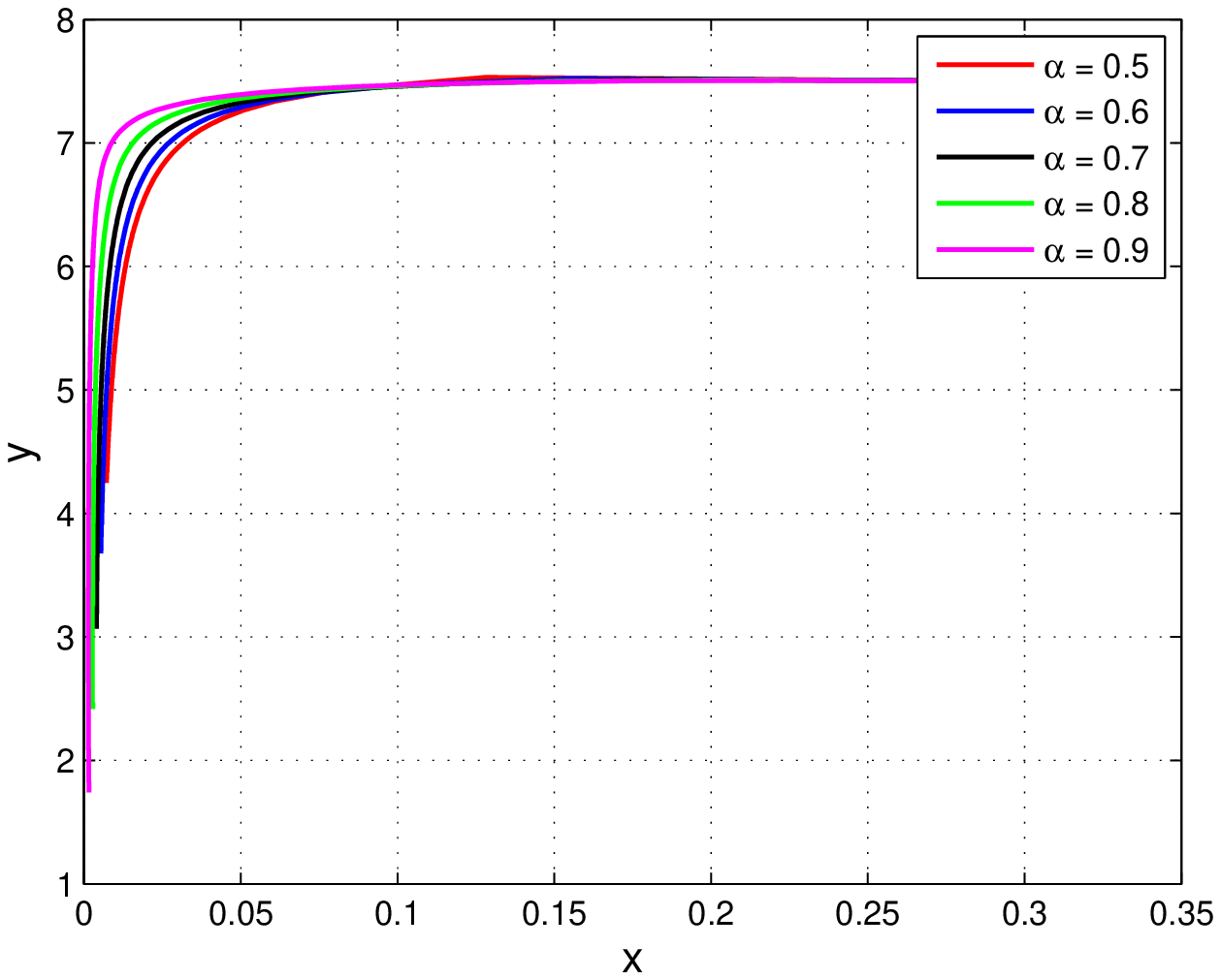}\\
        \footnotesize{(a) Model 1: $x_0 =15$, $y_0 =0.1$, $h=0.03$} & \footnotesize{(b) Model 2: $x_0 =0.3$, $y_0 =7.5$, $h=0.02$}\\
    \end{tabular}}
\end{figure}
\end{center}

\section{Comparison of complexity between the GL and NSFD scheme}

The complexity of the schemes are connected with the number of terms appearing in the discrete fractional GL derivative used for discretization of the fractional Caputo derivative. The number of terms is fixed by the value of the time step $h$. As a consequence it is important from the computational point of view to choose the time step $h$ as large as possible to obtain a good agreement with dynamical behavior of the system.
From this point of view the NSFD scheme seems to be competitive with GL scheme. Indeed, for all the examples we present in this paper we can observe that there is a factor of $10$ for the time step $h$ in order to obtain comparable results for both methods. More concretely, it means that we need to evaluate $\frac{9}{h}$ more quantities in the GL case to obtain a similar results as in the NSFD one. For example, for $\alpha=0.8$ on initial condition $(15,0.1)$ and time interval $[0,5]$ to obtain more or less comparable results with the NSFD scheme with $h=0.01$, which takes $0.36$ seconds, we need to consider the GL scheme with $h=0.001$ and in this case it takes $34.16$ seconds. This feature of NSFD methods gives possibility for numerical analysis of fractional differential problems on the large integration time, which is important due to the fact that the speed of convergence depends on the $\alpha$. It means that by reducing the value of $\alpha$ we need longer time to reach an equilibrium point.
\begin{center}
\begin{figure}[!h]
\centerline{%
    \begin{tabular}{cc}
        \includegraphics[width=0.48\textwidth]{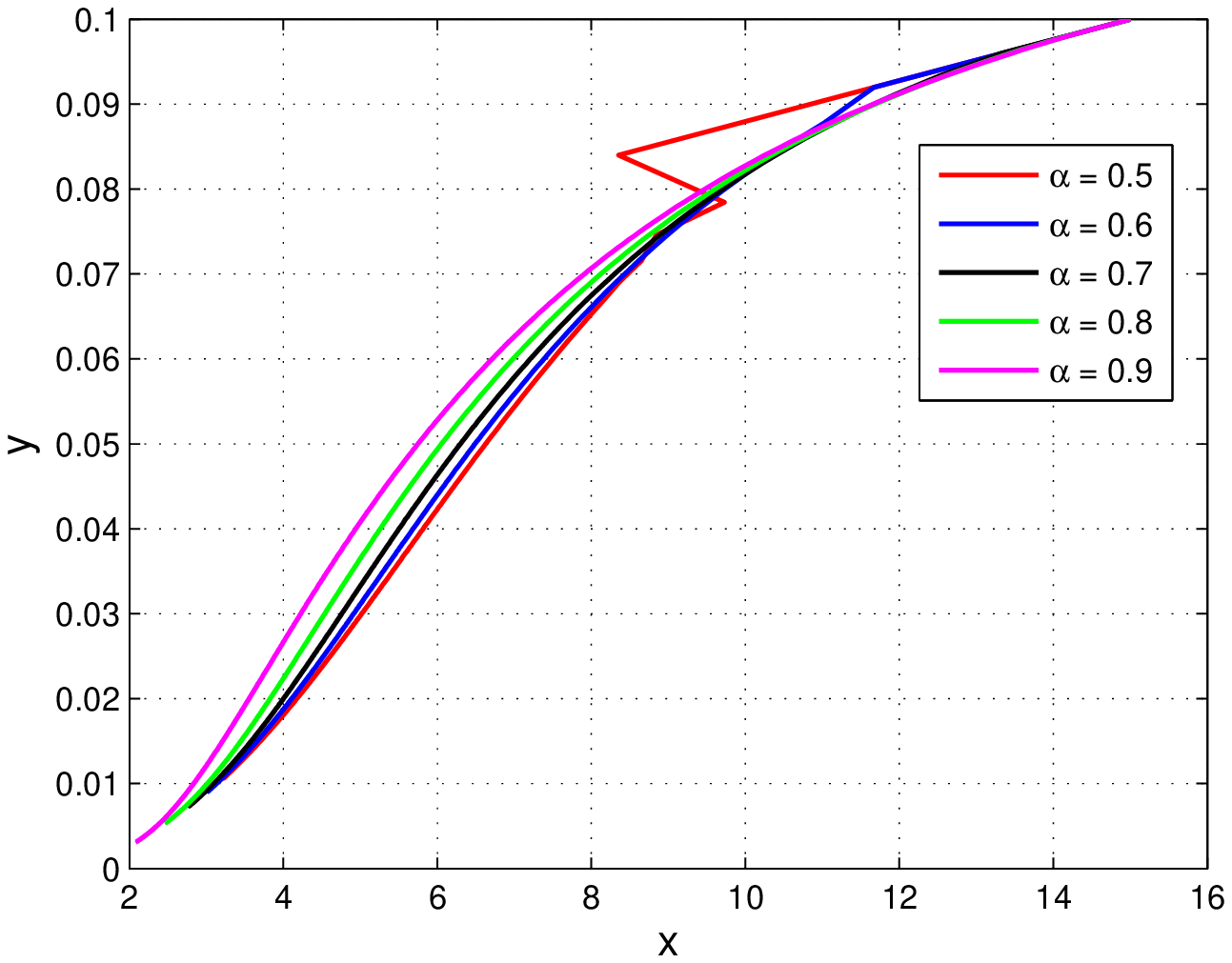} & \includegraphics[width=0.48\textwidth]{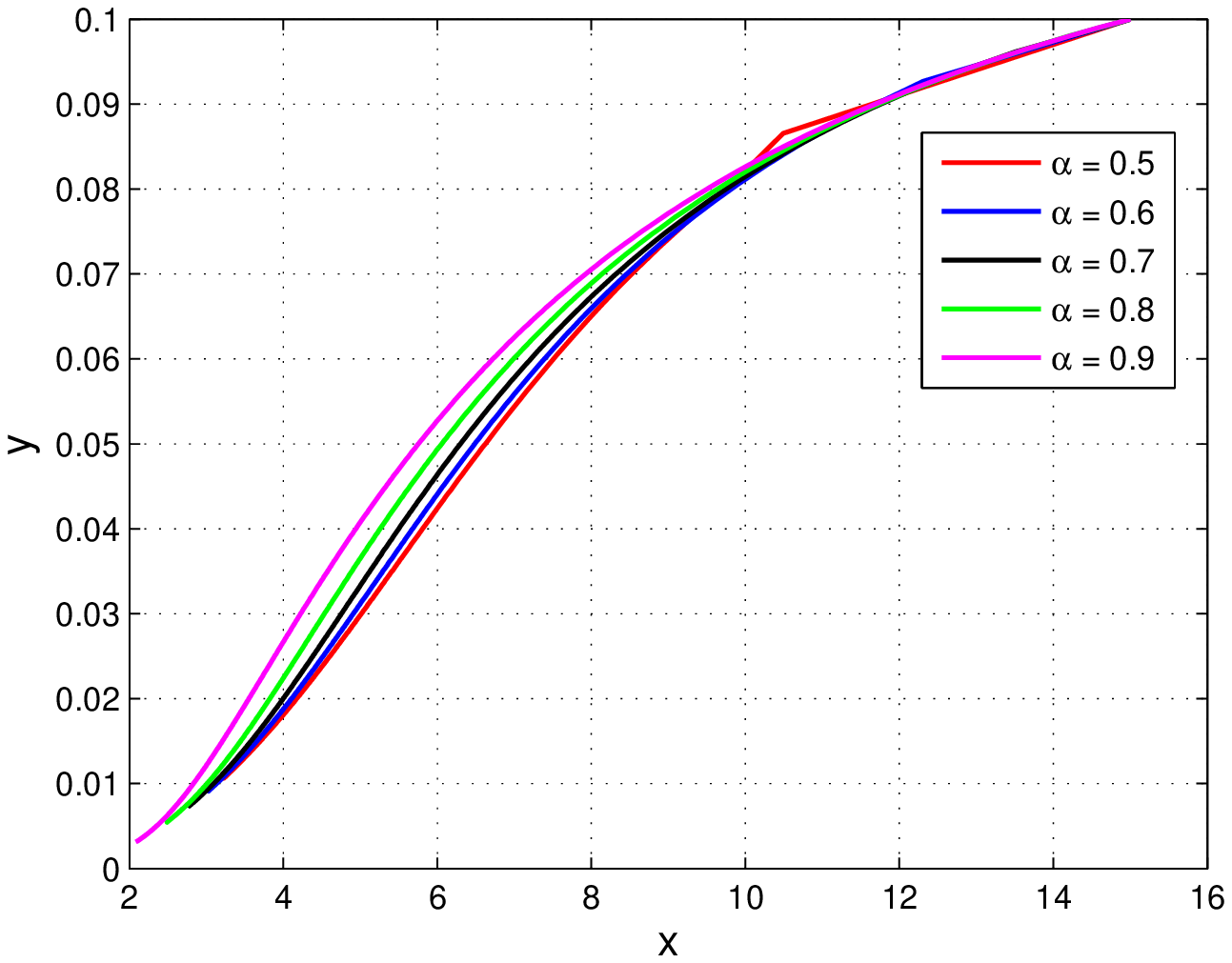}\\
        \footnotesize{(a)} & \footnotesize{(b)}
    \end{tabular}}
    \caption{Comparison of GL (a) and NSFD (b) method for numerical parameters $x_0 =15$, $y_0 =0.1$, $h=0.001$.}
\end{figure}
\end{center}



\newpage

\part{Application - a fractional predator-prey model}
\setcounter{section}{0}

In this Section, we give some applications of our numerical scheme to the fractional prey-predator model introduced by M. Javidi and al. \cite{JN}. In that paper, the authors generalize a classical prey-predator model replacing classical derivatives by Caputo fractional derivatives. Their system enters in our class of fractional systems and, as a consequence, preserves positivity (see Section \ref{model}). In order to illustrate some theoretical results included in the paper, the authors use a numerical method introduced by Atanackovic and Stankovich in (\cite{ata1},\cite{ata2}). This method consists in transforming first the fractional system into a classical one using some approximation formula for the fractional derivative. Second, to use a classical numerical method to integrate the differential system. In their paper, the authors use a Runge-Kutta method of $4-th$ order. It must be pointed out that there exists no convergence proof of this numerical method and that it deserves further studies as noticed in \cite{ata2}. This method is however effective, i.e. can be effectively implemented, but it can produce some dynamical artefact when the time step increment is not sufficiently small, in particular concerning the positivity property and the stability of the equilibrium point. To be more precise, the following numerical results of \cite{JN} do not correspond to the expected behavior of the solutions :

\begin{itemize}
\item In (Fig. 5, p. 8954) the positivity is not preserved.
\item In (Fig. 5, p. 8954) and (Fig. 3, p. 8952) the stability of the equilibrium point is not respected.
\item In (Fig. 1, p. 8950 and Fig. 3, p. 8952) the unicity of solutions is not satisfied.
\end{itemize}

In this part of the paper, we show that our method reproduces correctly the expected behavior of the solutions in all these cases for a lower computational demand. Indeed, the numerical method used in \cite{JN} uses the Runger-Kutta of order $4$ and some transformations which demand the numerical calculation of more quantities.

\section{Model introduction}
\label{model}

We consider the predator-prey interaction model with harvesting introduced in \cite{JN} and defined by
\begin{equation} \label{u}
\begin{split}
_{c}D_{0,t}^{\alpha}x(t)&=sx(t)\Big(1-\frac{x(t)}{K}\Big)-q\frac{x(t)y(t)}{1+q_1 x(t)},\\
_{c}D_{0,t}^{\alpha}y(t)&=\beta\frac{x(t)y(t)}{1+q_1 x(t)}-(s_0+E)y(t),
\end{split}
\end{equation}
where $0<\alpha\leq 1$. This model is a fractional generalization of the classical predator-prey model with Holling type-II interaction and nonconstant predator harvesting
\begin{equation} \label{classical}
\begin{split}
\frac{d x(t)}{dt}&=sx(t)\Big(1-\frac{x(t)}{K}\Big)-q\frac{x(t)y(t)}{1+q_1 x(t)},\\
\frac{d y(t)}{dt}&=\beta\frac{x(t)y(t)}{1+q_1 x(t)}-s_0y(t)-Ey(t),
\end{split}
\end{equation}
where $x(t)$ and $y(t)$ denote the prey and predator population densities, respectively. The harvesting policy in the above model involves a linear harvesting rate in terms of the predator species: $Ey(t)$, where $E\geq 0$ represents a constant harvesting effort for the predator. In other words, when the abundance of the predator species increases, then the number of predator harvested will increase linearly. The average per predator death rate is denoted as a constant $s_0$ and the positive parameter $\beta$ is the conversion factor denoting the number of newly born predators for each captured prey.

The model assumes that a predator spends its time searching for prey or prey handling (chasing, killing, eating and digesting), i.e. that the consumption of prey is limited because of the time that predator needs to capture, consume, and digest every particular prey. This limitations are covered in the model by the Holling type-II functional response: $\frac{qx(t)}{1+q_1 x(t)}$, where positive constants $q$ and $q_1$ denote the maximal predator consumption rate (unit: 1/time) and handling time (unit: 1/prey), respectively \cite{freedman}. In the case when the handling time $q_1$ is equal zero, which is equivalent to the situation that individual predators have an unlimited ability to catch and consume prey, we obtain the Holling type-I functional response appearing in the well known Lotka-Volterra model.

Moreover, in the absence of predation, the prey host population satisfies logistic growth, i.e. $\frac{dx(t)}{dt}=sx(t)\Big(1-\frac{x(t)}{K}\Big)$, where $K$ is carrying capacity of the pray in some closed community in the absence of predator and harvesting, the constant $s$ denotes the intrinsic growth rate for prey population.

The construction of model (\ref{u}) allows to consider it as an example of our general class of fractional differential equations (\ref{frac}) and, as a consequence, based on Theorem \ref{FPT}, we have :

\begin{twr}
The fractional prey-predator system (\ref{u}) preserves positivity.
\end{twr}


As in \cite{JN}, we study the equilibrium points of the system (\ref{u}) in order to compare with the results obtained in \cite{JN}. We will observe how the solutions behave near the equilibrium points in terms of positivity, stability or unicity. \\

The equilibrium points of (\ref{u}) are solutions of the system (see II,2,2.1) :
$$
\begin{array}{ll}
f_1(x,y)=&sx(t)\Big(1-\frac{x(t)}{K}\Big)-q\frac{x(t)y(t)}{1+q_1 x(t)}=0,\\
f_2(x,y)=&\beta\frac{x(t)y(t)}{1+q_1 x(t)}-s_0y(t)-Ey(t)=0.
\end{array}
$$
We deduce that there exists the trivial and semi-trivial intersection points for any choice of positive values of parameters: the origin point $P_0=(0,0)$  and the predator extinction point $P_1=(K,0)$. The interior equilibrium point
$$
P_2=\left(\frac{K}{R_0+q_1K(R_0-1)},\frac{sR_0(1+q_1K)^2(R_0-1)}{q(R_0+q_1K(R_0-1))^2}\right)=(x_{\ast},y_{\ast}),
$$
where $R_0=\frac{\beta K}{(1+q_1K)(s_0+E)}$, exists if $R_0>1$ and $x_{\ast}\leq K$. It means that at the point $P_2$ the predator and prey coexist under certain conditions.

The results of the stability analysis of fixed points presented in \cite{JN} are summarized in the statement below.

\begin{twr}
\label{tw:1}
The equilibrium point
\begin{itemize}
\item[$\bullet$] $P_0$ is a saddle point,
\item[$\bullet$] $P_1$ is locally asymptotically stable if $R_0<1$ and unstable if $R_0>1$,
\item[$\bullet$] $P_2$ is locally asymptotically stable if $\alpha<\tilde{\alpha}$ and unstable if $\alpha>\tilde{\alpha}$, where $\tilde{\alpha}$ is a marginal value equals $\frac{2}{\pi}|\arg{(\lambda_i)}|$, $i=1,2$.
\end{itemize}
\end{twr}

In above $\lambda_i$, $i=1,2$ are eigenvalues of the Jacobian matrix
$$
J(P_2)=\left[\begin{array}{cc} \frac{\partial f_1}{\partial x} & \frac{\partial f_1}{\partial y}\\
\frac{\partial f_2}{\partial x} & \frac{\partial f_2}{\partial y}\end{array}\right]_{(x,y)=(x_{\ast},y_{\ast})}.
$$

In the following, we will use this Theorem in order to provide for each numerical simulations the conditions for stability.

\section{The NSFD scheme for prey-predator model}

According to the Theorem \ref{thw1} we are able to construct for the problem (\ref{u}) the nonstandard finite difference scheme of fractional order which preserves positivity. The method has the following form
\begin{equation}
\label{ud}
\begin{split}
\frac{1}{h^{\alpha}}\sum_{j=0}^n \alpha_j^{[\alpha]}(x_{n-j}-x_0)& =sx_{n-1}-\Big(\frac{s}{K}x_{n-1}+\frac{q y_{n-1}}{1+q_1 x_{n-1}}\Big)x_{n},\\
\frac{1}{h^{\alpha}}\sum_{j=0}^n \alpha_j^{[\alpha]}(y_{n-j}-y_0)& =\frac{\beta x_{n-1}}{1+q_1 x_{n-1}} y_{n-1}-(s_0+E)y_n,
\end{split}
\end{equation}
where $\alpha_0^{[\alpha]}, \alpha_j^{[\alpha]}$ are given by (\ref{lam}).
Left hand sides of system (\ref{ud}) can be rewritten as follows
\begin{equation} \label{ud1}
\begin{split}
x_n+\sum_{j=1}^n\alpha_j^{[\alpha]}x_{n-j}-x_0\sum_{j=0}^n\alpha_j^{[\alpha]}&=h^{\alpha}sx_{n-1}-h^{\alpha}\Big(\frac{s}{K}x_{n-1}+\frac{q y_{n-1}}{1+q_1 x_{n-1}}\Big)x_{n},\\
y_n+\sum_{j=1}^n\alpha_j^{[\alpha]}y_{n-j}-y_0\sum_{j=0}^n\alpha_j^{[\alpha]}&=h^{\alpha}\frac{\beta x_{n-1}}{1+q_1 x_{n-1}} y_{n-1}-h^{\alpha}(s_0+E)y_n.
\end{split}
\end{equation}
According to definition (\ref{lam}) of $\alpha_j^{[\alpha]}$ it is easily seen that the sum of this coefficient is positive, i.e.
$$
\sum_{j=0}^n\alpha_j^{[\alpha]}=(1-\alpha)(2-\alpha)\cdots(n-1-\alpha)\frac{n-\alpha}{n!}>0.
$$
and
$$
(1-\alpha)(2-\alpha)\cdots(n-1-\alpha)\frac{n-\alpha}{n!}=\frac{1}{n!}\prod_{k=0}^{n-1}(k-(\alpha-1))=\alpha_n^{[\alpha-1]}.
$$
This leads to the explicit formula for the approximate values of prey and predator
\begin{equation} \label{ud1}
\begin{split}
x_n&=\frac{Ku_{n-1}\left[h^{\alpha}s x_{n-1}-\sum_{j=1}^n\alpha_j^{[\alpha]} x_{n-j}+x_0\alpha_n^{[\alpha-1]}\right]}{Ku_{n-1}+h^{\alpha}\Big(su_{n-1}x_{n-1}+qKy_{n-1}\Big)},\\
y_n&=\frac{h^{\alpha}\beta x_{n-1}y_{n-1}-u_{n-1}\sum_{j=1}^n\alpha_j^{[\alpha]} y_{n-j}+u_{n-1}y_0\alpha_n^{[\alpha-1]}}{(1+ h^{\alpha}(s_0+E))u_{n-1}},
\end{split}
\end{equation}
where $u_n=1+q_1 x_{n}$.

\section{Numerical simulations}

In this section, using our numerical scheme we provide simulations and compare it with the corresponding one provided in \cite{JN}. The time step is chosen to be $h=0.01$ as in (\cite{JN}, Section 5 p. 8954) or higher (for example $h=0.1$) when the numerical result is already better than the one obtained in \cite{JN}.

\subsection{Lost of unicity}

Numerical scheme often loose unicity of solution when the time step increment is too big. This is the case in the examples of simulations provided by Javidi and al. in \cite{JN} (Fig. 1, p. 8950 and Fig. 3, p. 8952) with a time step increment $h=0.01$. Using nonstandard finite difference scheme, we observe similar patterns but respecting unicity of solutions. Moreover, the results are already concordant for at least ten times higher time step.

\begin{center}
\begin{figure}[!h]
\centerline{%
    \begin{tabular}{cc}
        \includegraphics[width=0.48\textwidth]{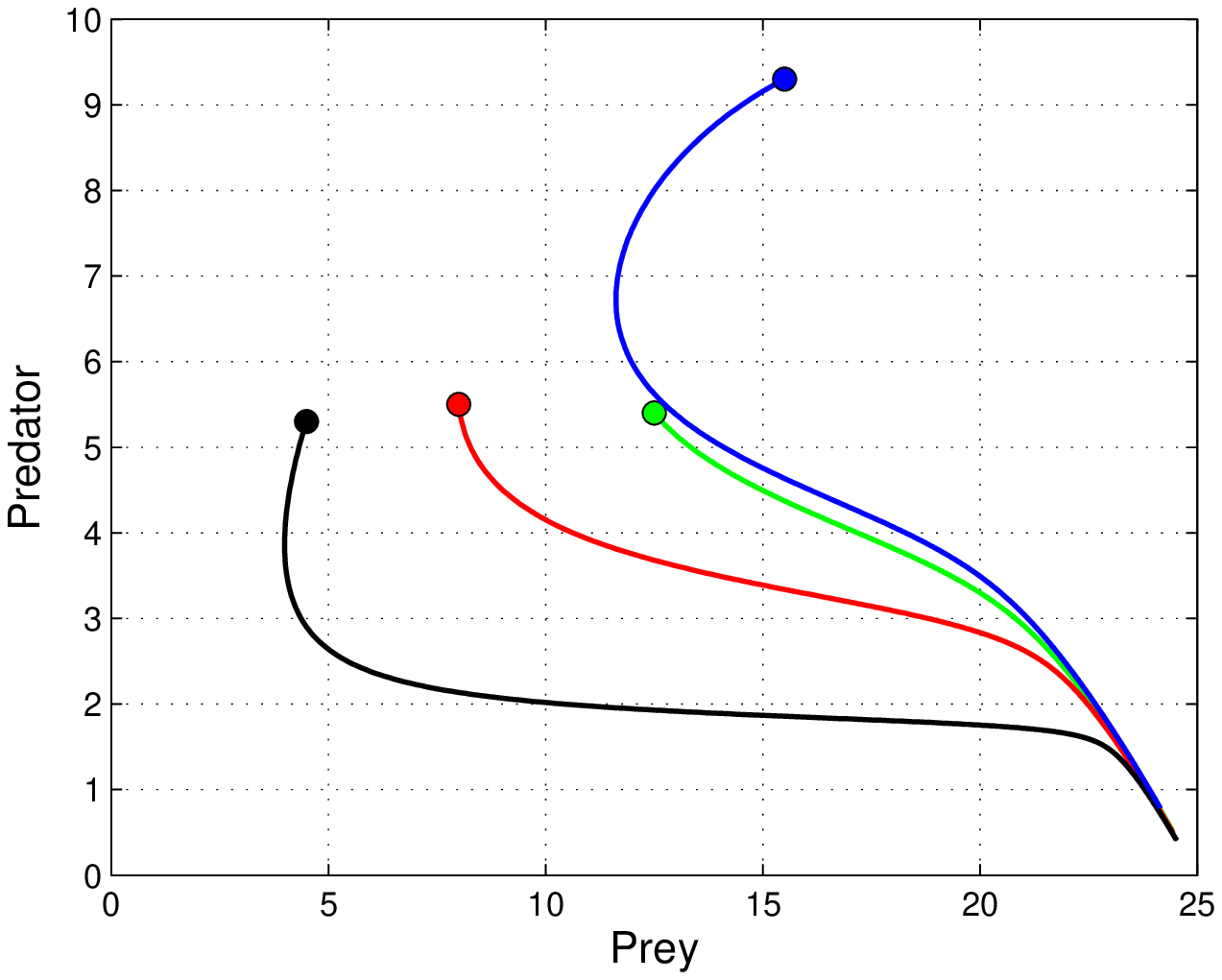} & \includegraphics[width=0.48\textwidth]{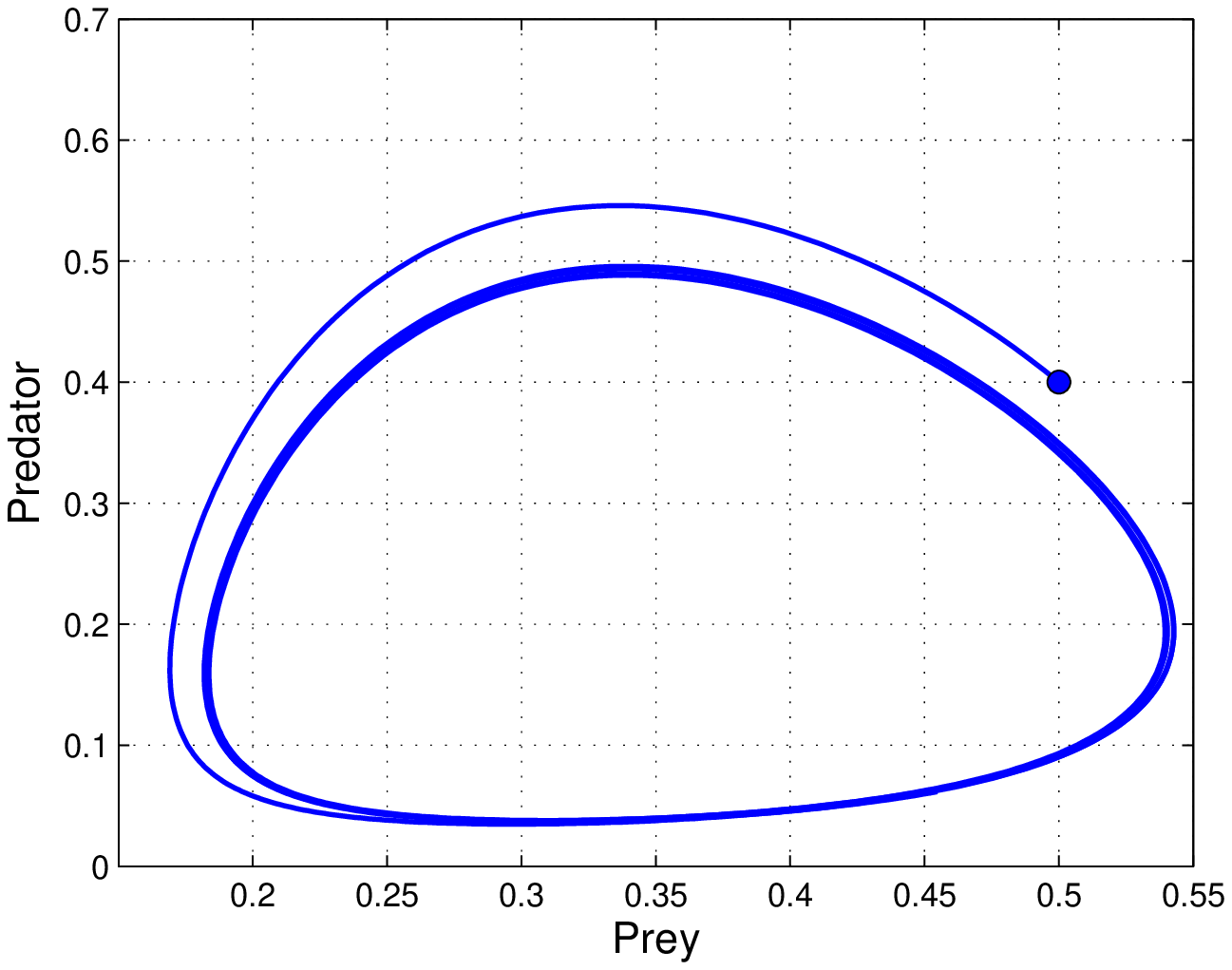}\\
        \footnotesize{(a) Comparison with Fig. 1, p. 8950 in \cite{JN}} & \footnotesize{(b) Comparison with Fig. 3, p. 8952 in \cite{JN}}\\
    \end{tabular}}
    \caption{Results obtained by the method (\ref{met}), $h=0.1$.}
\end{figure}
\end{center}

\subsection{Lost of positivity}
\label{lost-positivity}

The property of positivity for numerical solutions obtained by standard discretization can be disturbed. This occurrence can be observed in the example presented in (\cite{JN}, Fig. 5, $\alpha=0.65$) with parameters $s=0.2$, $K=25$, $q=1$, $q_1=0.1$, $\beta=2$, $s_0=0.5$, $E=1.3$. For this set of parameters the threshold for the stability of the equilibrium point $P_2=(0.9890,0.2111)$ is given by $R_0=7.9365$. The marginal value is $\tilde{\alpha}=0.9947$ and initial conditions $(6.5,5.4)$, $(8,5.5)$, $(4.5,4.3)$, $(8.7,2.3)$. In front of results presented in \cite{JN} we give ones obtained by the nonstandard numerical method (\ref{met}) with the same parameters. For comparison purposes we put in our calculations the same time step ($h=0.01$) and ten times larger ($h=0.1$) than in \cite{JN}. Nevertheless, we receive results which are compatible with the expected dynamical behavior.
\begin{center}
\begin{figure}[!h]
\centerline{%
    \begin{tabular}{cc}
        \includegraphics[width=0.48\textwidth]{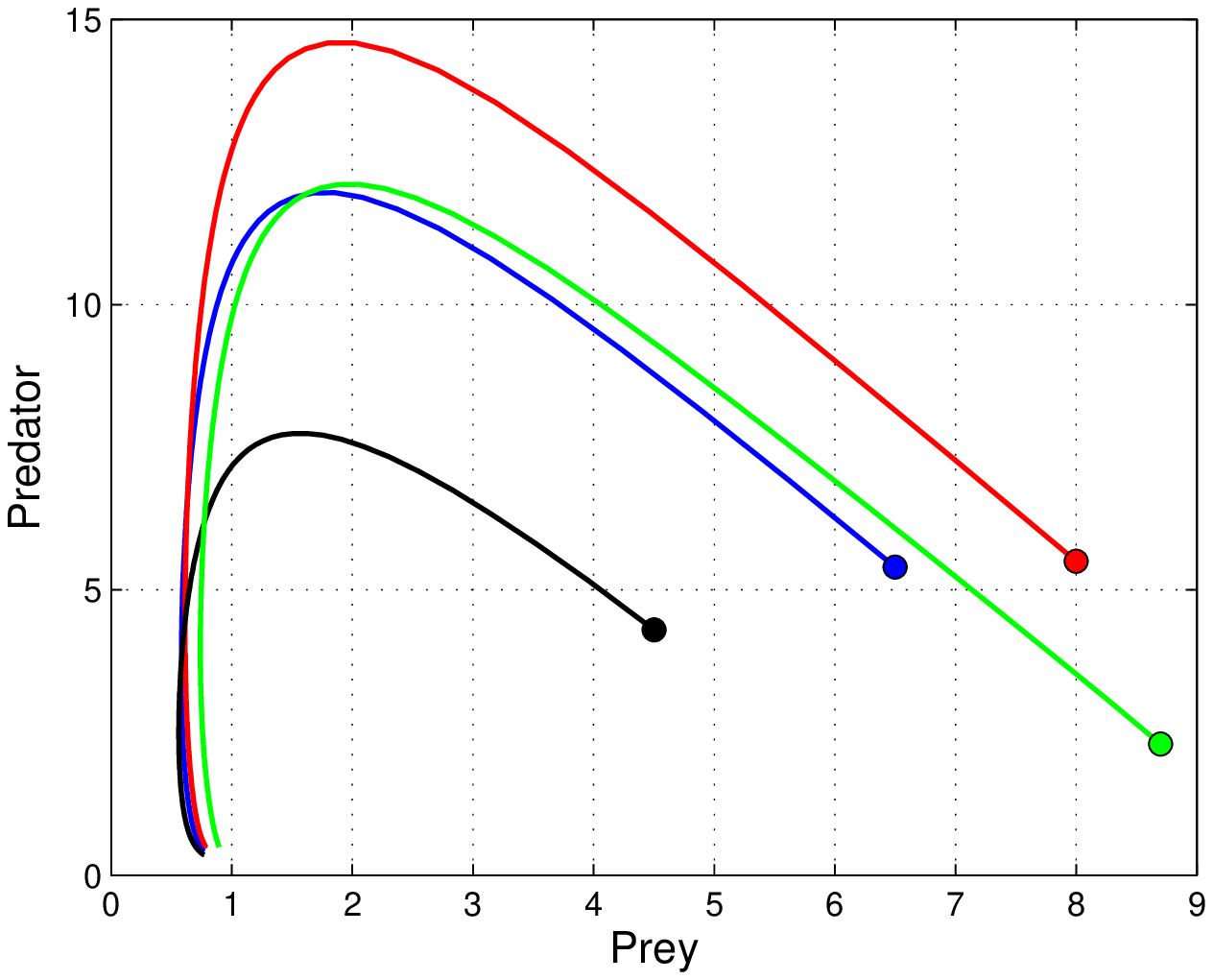} & \includegraphics[width=0.48\textwidth]{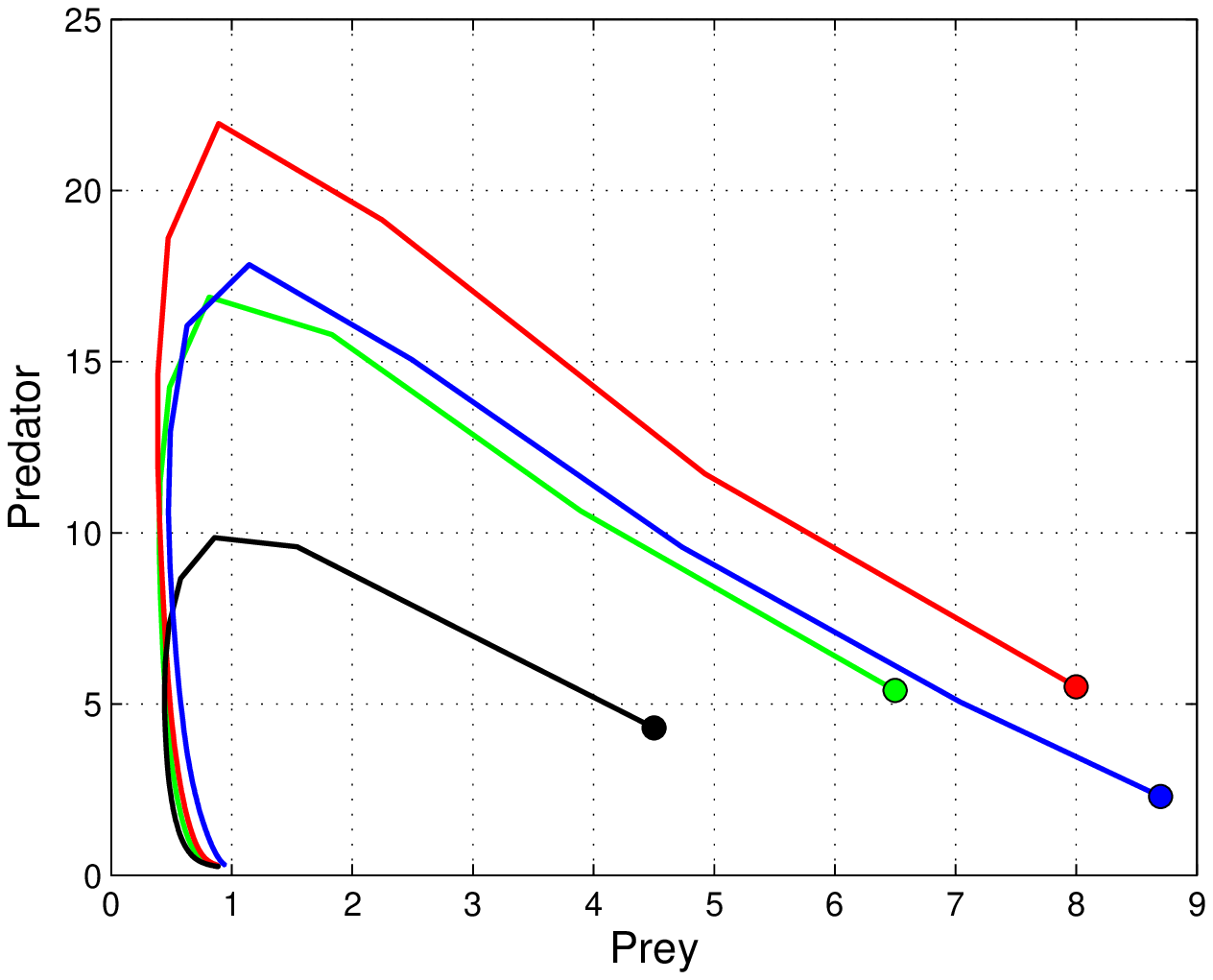}\\
        \footnotesize{(a) $h=0.01$} & \footnotesize{(b) $h=0.1$}
    \end{tabular}}
    \caption{Results obtained by the method (\ref{met}).}
\end{figure}
\end{center}

\subsection{Lost of stability}

 In the results presented in (\cite{JN}, Fig. 3(a)) for model (\ref{u}) with parameters $s=0.1$, $K=25$, $q=1$, $q_1=2$, $\beta=5$, $s_0=0.7$, $E=0.3$ we can observe a lost of stability of the numerical solutions. In this case the threshold for the stability of the equilibrium point $P_2=(0.3333,0.1644)$ is given by $R_0=2.4510$. The marginal value is $\tilde{\alpha}=0.9501$ and the initial condition is $(0.5,0.4)$.

 By Theorem \ref{tw:1}, the fixed point $P_2$ is stable for values of $\alpha$ less than the marginal value $\tilde{\alpha}$. However, the figure 3(a) in \cite{JN} shows an unstable behavior which is not coherent with the theoretical result. Using our scheme, we obtain for all values of $\alpha$ a simulation which is coherent with the expected dynamical behavior according to Theorem \ref{tw:1}.

\begin{center}
\begin{figure}[h]
\centerline{%
    \begin{tabular}{cc}
        \includegraphics[width=0.48\textwidth]{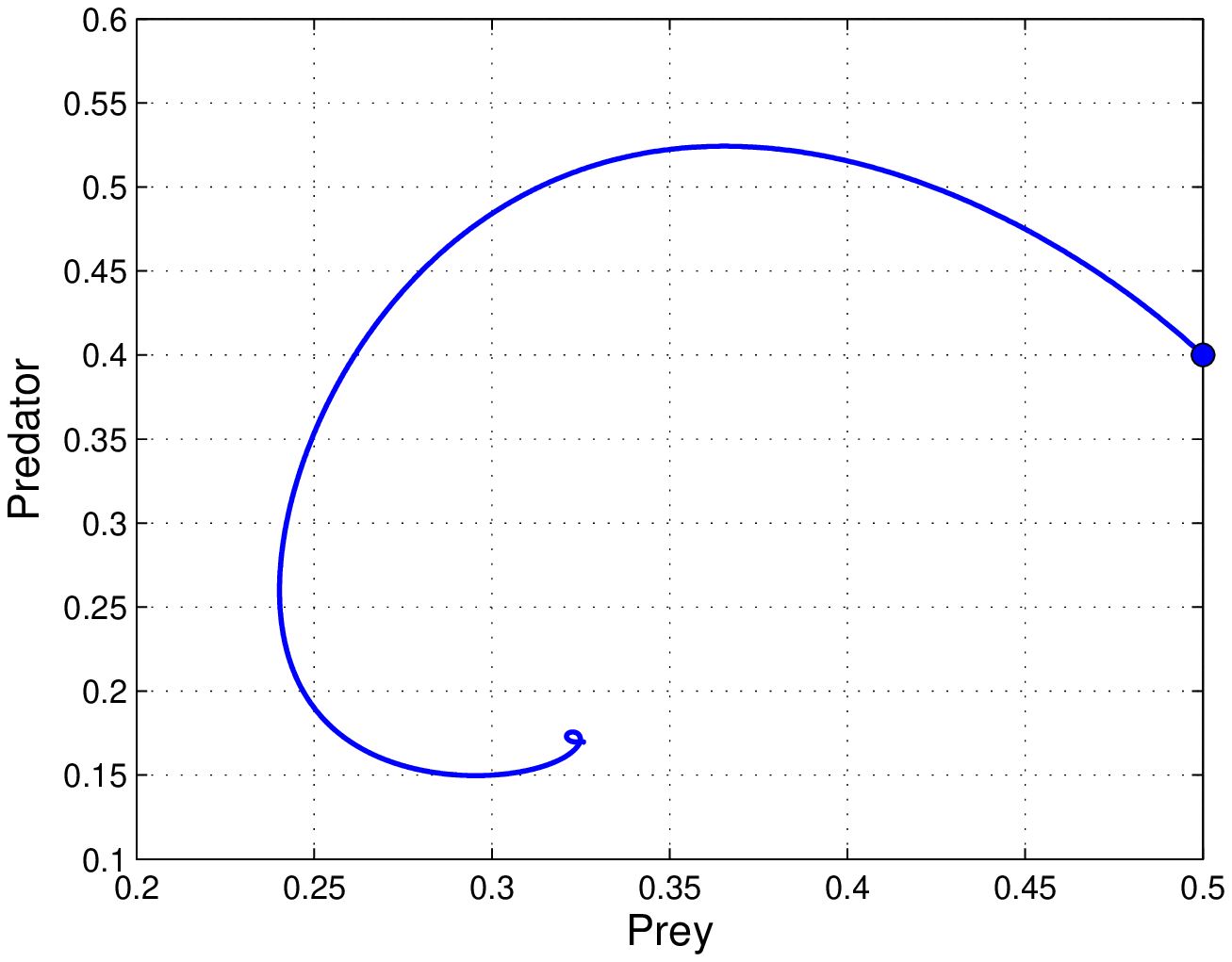} & \includegraphics[width=0.48\textwidth]{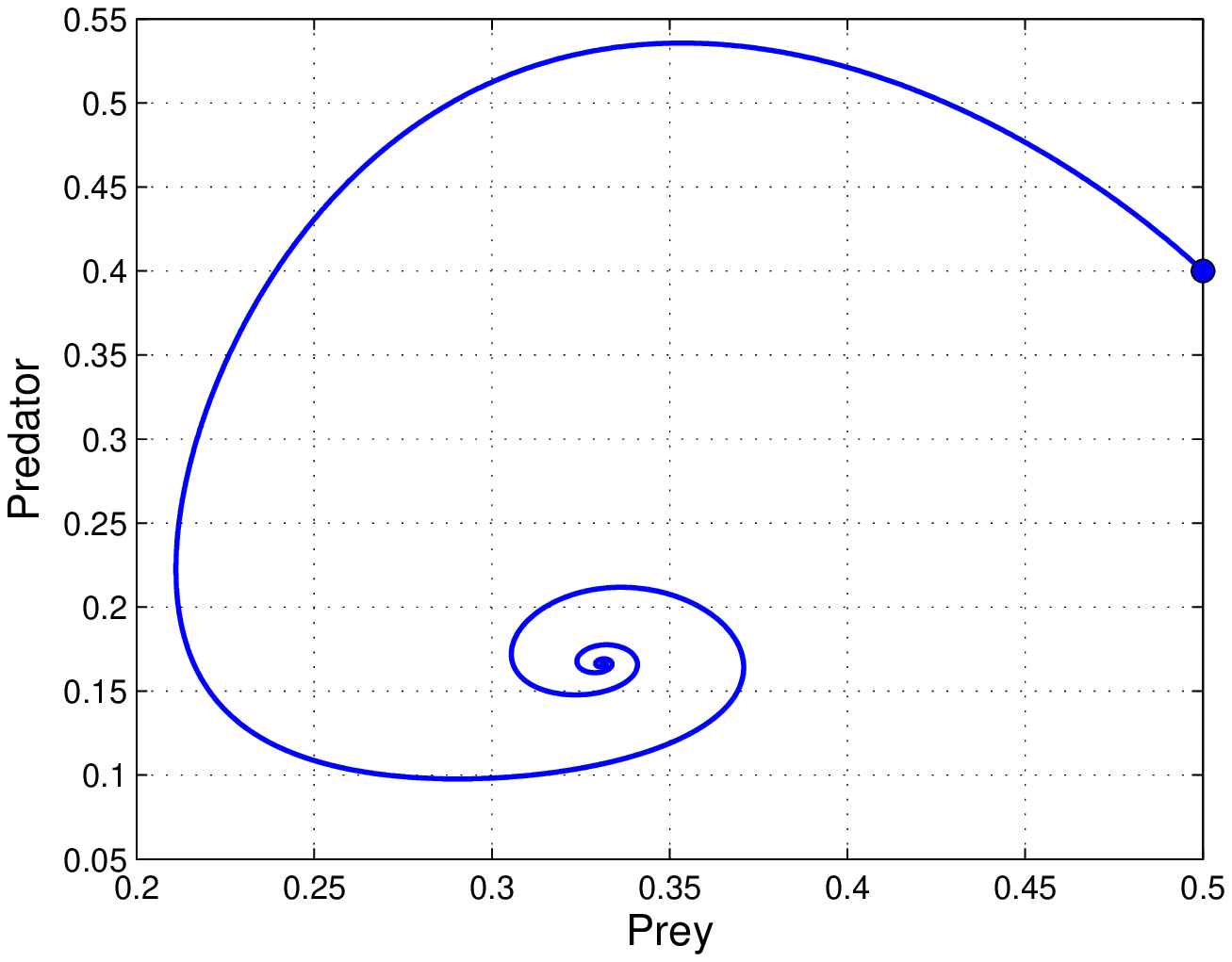}\\
        \footnotesize{(a) $\alpha=0.75$} & \footnotesize{$\alpha=0.85$}\\
    \end{tabular}}
\caption{Results obtained by the method (\ref{met}).}
\end{figure}
\end{center}

\subsection{More numerical results}

The previous numerical simulations are given in order to show that our numerical scheme solve many numerical problems arising in simulations of fractional systems. In this last Section, we provide more simulations which can be directly compared with the corresponding one in \cite{JN} but with a at least ten times bigger time increment, proving the efficiency of our scheme. As we can observe, all presented simulations confirm compliance with the theoretical analysis of the dynamical behavior of solutions of the model (\ref{u}).

\subsubsection{Behavior near the equilibrium point $P_1$}

We provide numerical simulations for the behavior of the fractional system near the equilibrium point $P_1$ which is always stable. Different values of $\alpha$ are tested. We provide first the phase portrait as well as the individual behavior of each variables. The set of parameters: $s=0.5$, $K=5$, $q=1$, $q_1=2$, $\beta=0.02$, $s_0=0.7$, $E=0.3$. The threshold for the stability of the equilibrium point $P_1=(5,0)$ is given by $R_0=0.0091$. The time step increment is $h=0.5$, the integration time is $T=500$ and as the initial point we take $(0.5,0.4)$.

\begin{center}
\begin{figure}[!h]
\centerline{%
    \begin{tabular}{c}
        \includegraphics[width=0.48\textwidth]{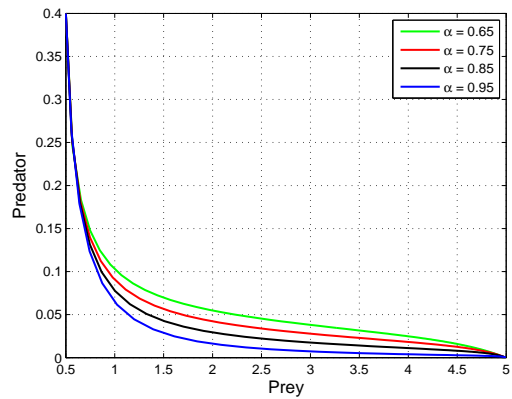}\\
        \footnotesize{(a) Phase portrait}\\
    \end{tabular}}
%
\centerline{%
    \begin{tabular}{cc}
        \includegraphics[width=0.48\textwidth]{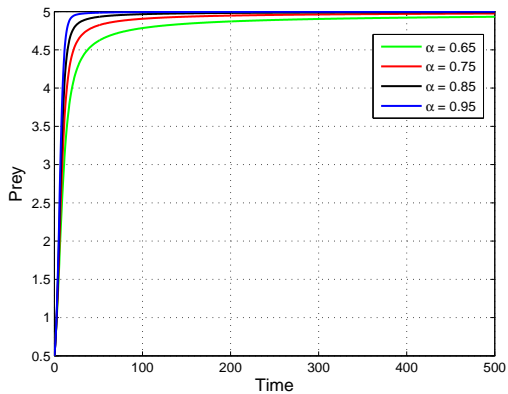} & \includegraphics[width=0.48\textwidth]{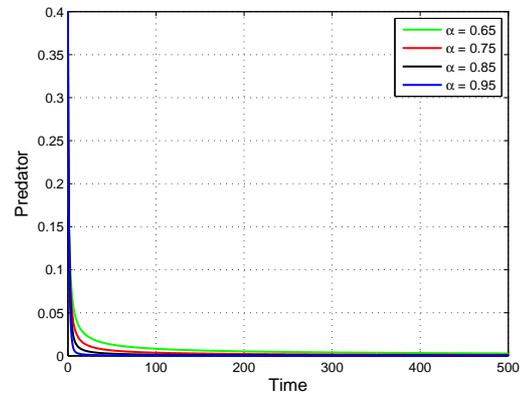}\\
        \footnotesize{(b) Density of prey} & \footnotesize{(c) Density of predator}\\
    \end{tabular}}
    \caption{Numerical simulations for different values of order $\alpha$.}
\end{figure}
\end{center}
\paragraph{Dynamical behavior of each variables near the equilibrium point $P_2$}
\label{dyn_beh}

In \cite{JN}, Javidi and al, the authors provide the dynamical behavior of each variables (p. 8953, Fig. 4).  Due to the change of stability produced by the numerical method, they observe densities for predator and prey which do not converge to the equilibrium point. In our case, with a bigger time step, we obtain a very good agreement with the theoretical expected behavior.

In order to observe the dynamical behavior of the numerical solutions  obtained by implementing the NSFD scheme, we consider the set of parameters: $s=0.1$, $K=5$, $q=1$, $q_1=2$, $\beta=5$, $s_0=0.7$, $E=0.3$. The threshold for the stability of equilibrium point $P_2=(0.3333,0.1556)$ is given by $R_0=2.2727$. The marginal value is $\tilde{\alpha}=0.9587$. The time step is $h=0.5$ and the integration time is $T=500$. The initial point is $(0.5,0.4)$.

\begin{center}
\begin{figure}[!h]
\centerline{%
    \begin{tabular}{cc}
        \includegraphics[width=0.48\textwidth]{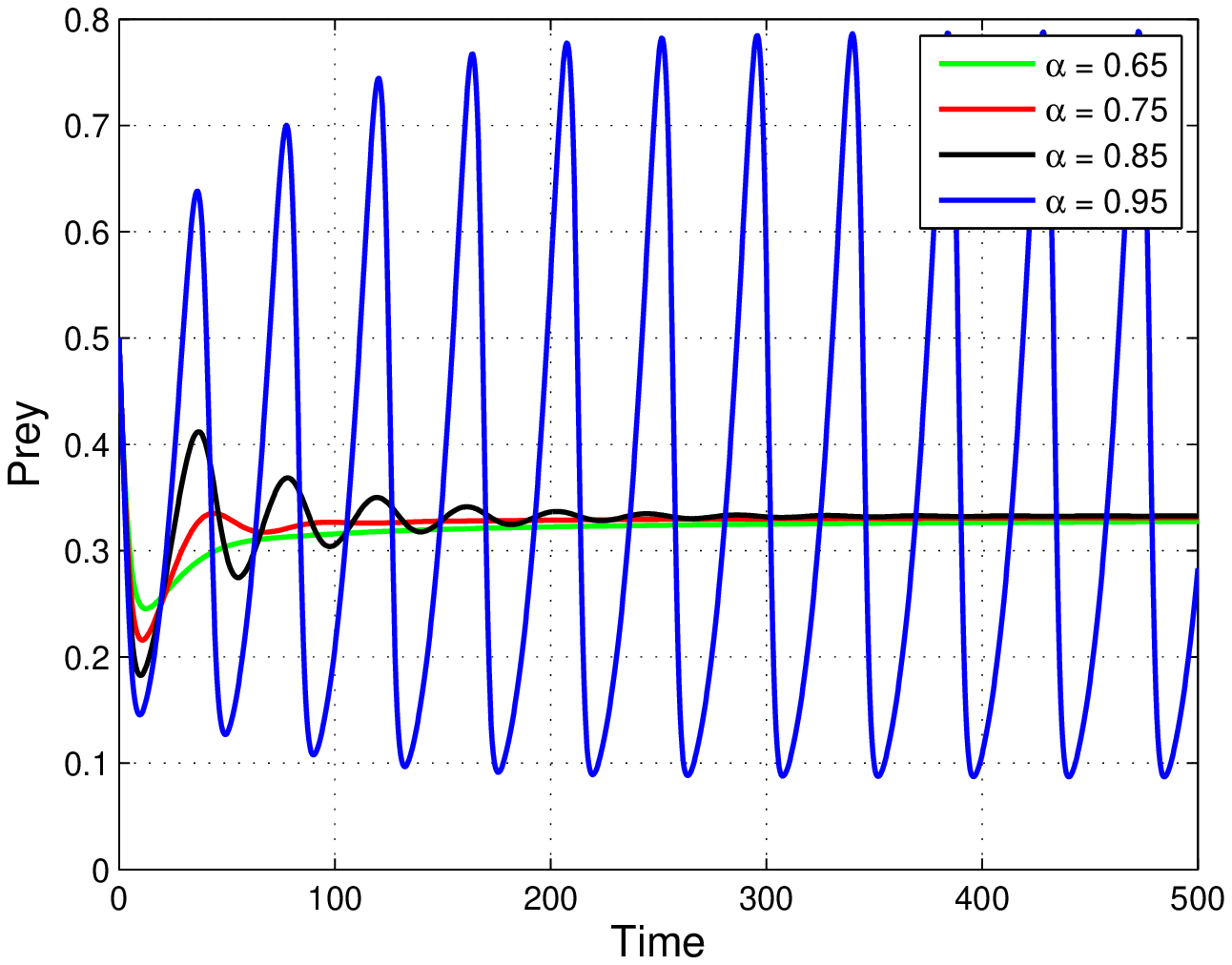} & \includegraphics[width=0.48\textwidth]{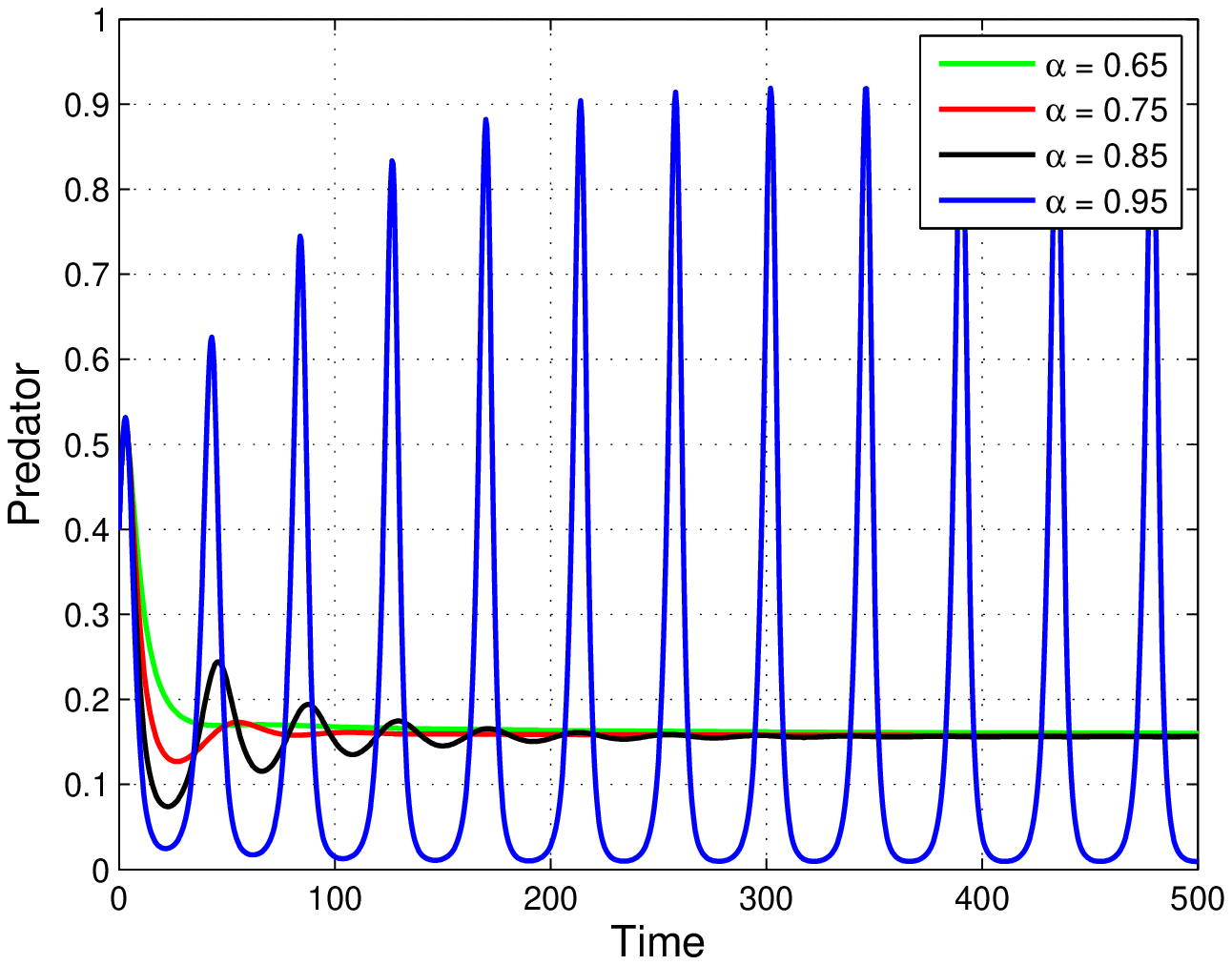}\\
    \end{tabular}}
    \caption{Densities of prey and predator respectively, for different values of order $\alpha$.}
\end{figure}
\end{center}

We provide also the simulations with parameters given by $s=5$, $K=5$, $q=0.1$, $q_1=2$, $\beta=4$, $s_0=0.5$, $E=0.3$ in order to compare with the figure presented in (\cite{JN}, Fig. 2, p. 8951). In this case, the threshold for the stability of the equilibrium point $P_2=(0.3333,77.7778)$ is given by $R_0=2.2727$. The marginal value is $\tilde{\alpha}=0.6576$. The time step is $h=0.5$ and the integration time is $T=300$. We take the initial point $(2.5,4.4)$.

\begin{center}
\begin{figure}[!h]
\centerline{%
    \begin{tabular}{c}
        \includegraphics[width=0.48\textwidth]{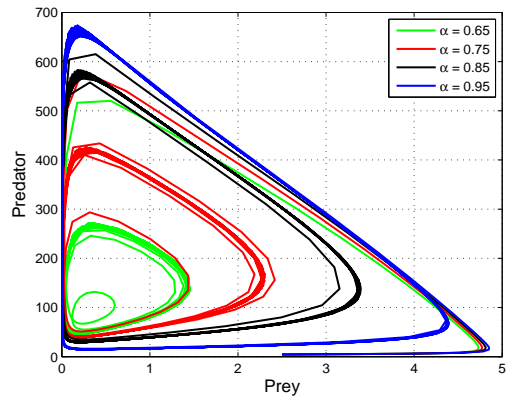}\\
        \footnotesize{(a) Phase portrait}\\
    \end{tabular}}
%
\centerline{%
    \begin{tabular}{cc}
        \includegraphics[width=0.48\textwidth]{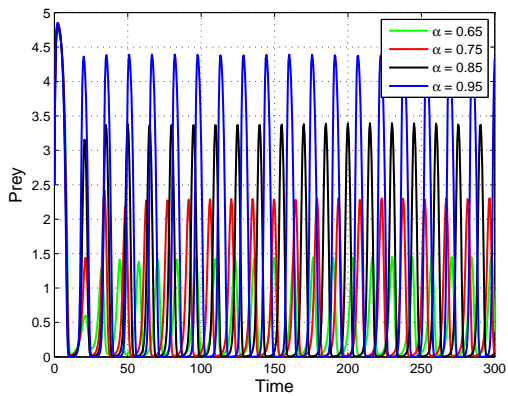} & \includegraphics[width=0.48\textwidth]{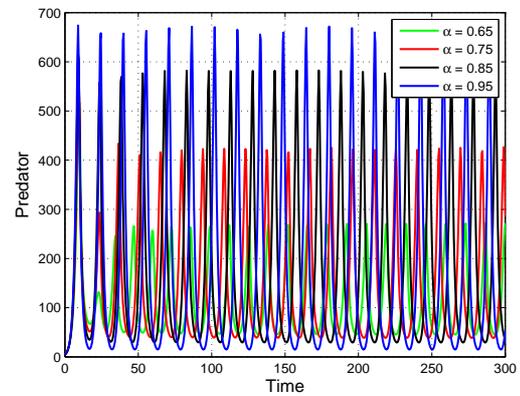}\\
        \footnotesize{(b) Density of prey} & \footnotesize{(c) Density of predator}\\
    \end{tabular}}
    \caption{Numerical simulations for different values of order $\alpha$. }
\end{figure}
\end{center}
The phase portrait possesses a good agreement with the theoretical result even if the time step increment if very big ($h=0.5$).

We have also tested what is the behavior of the fractional system when the predation is very high. We observe the following behavior starting with two species having almost the same density : the predator eat most of the prey at the beginning and then we have a rapid decreasing of the predator to an equilibrium level which is compatible with the  survival of enough prey to ensure the existence of predators. The results are presented on Figure 17. Numerical simulations are provided for different values of order $\alpha$ with parameters $s=0.1$, $K=5$, $q=1$, $q_1=2$, $\beta=15$, $s_0=0.7$, $E=0.3$. The threshold for equilibrium point $P_2=(0.0769,0.1136)$ is given by $R_0=6.8182$. The marginal value is $\tilde{\alpha}=0.9874$. The time step $h=0.2$ and the integration time is $T=300$. The initial point is $(0.5,0.4)$.

\begin{center}
\begin{figure}[!h]
\centerline{%
    \begin{tabular}{c}
        \includegraphics[width=0.48\textwidth]{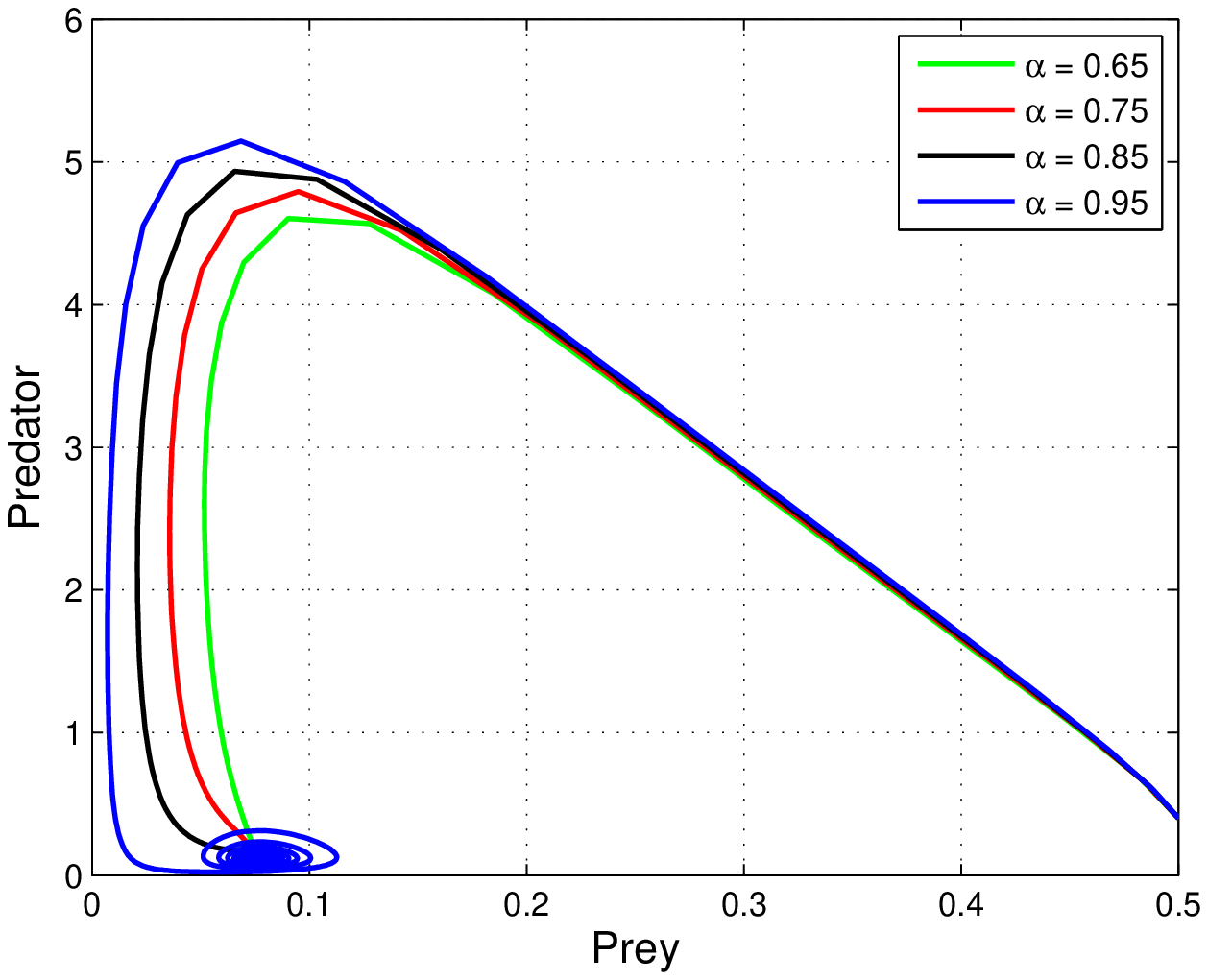}\\
        \footnotesize{(a) Phase portrait}\\
    \end{tabular}}
\centerline{%
    \begin{tabular}{cc}
        \includegraphics[width=0.48\textwidth]{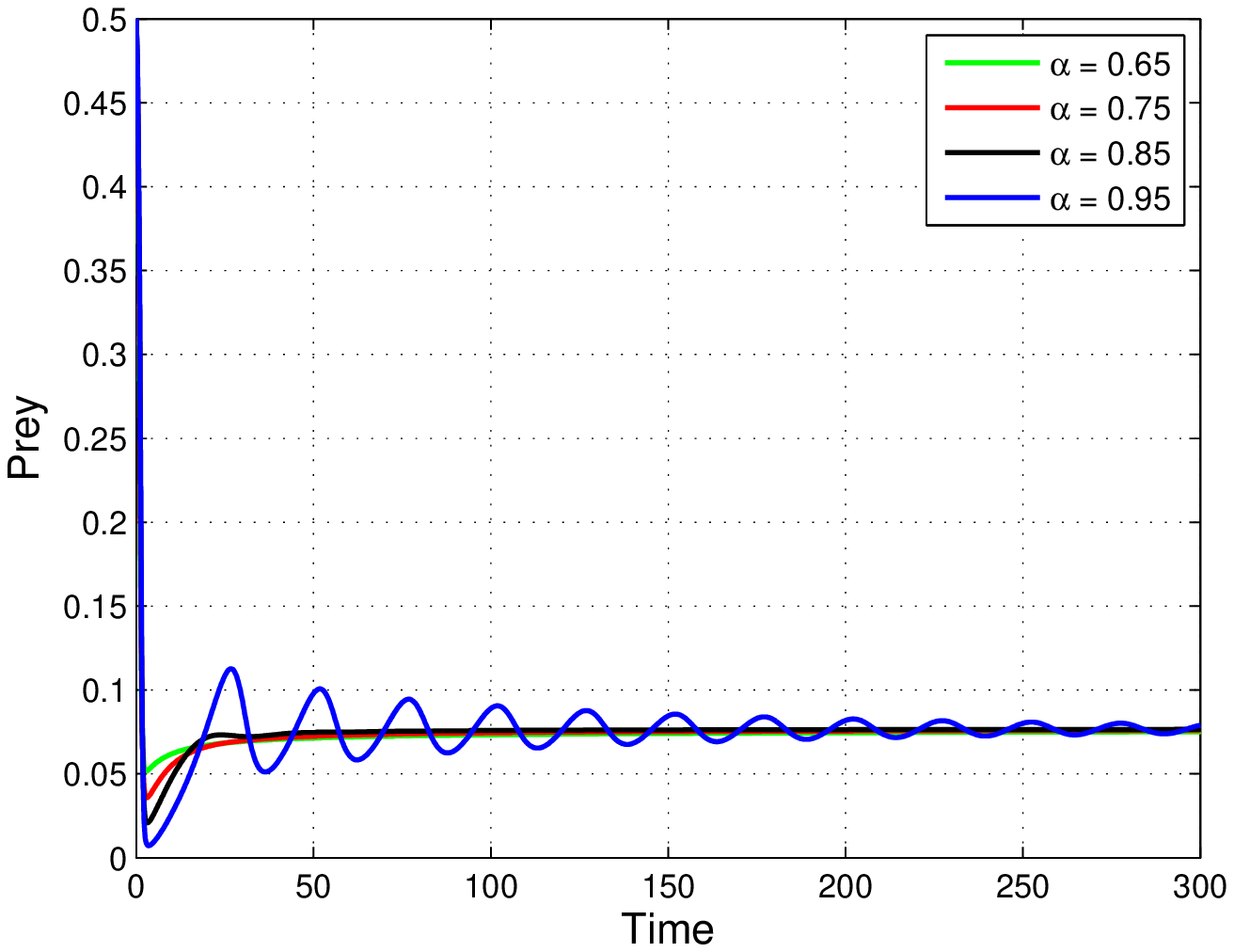} & \includegraphics[width=0.48\textwidth]{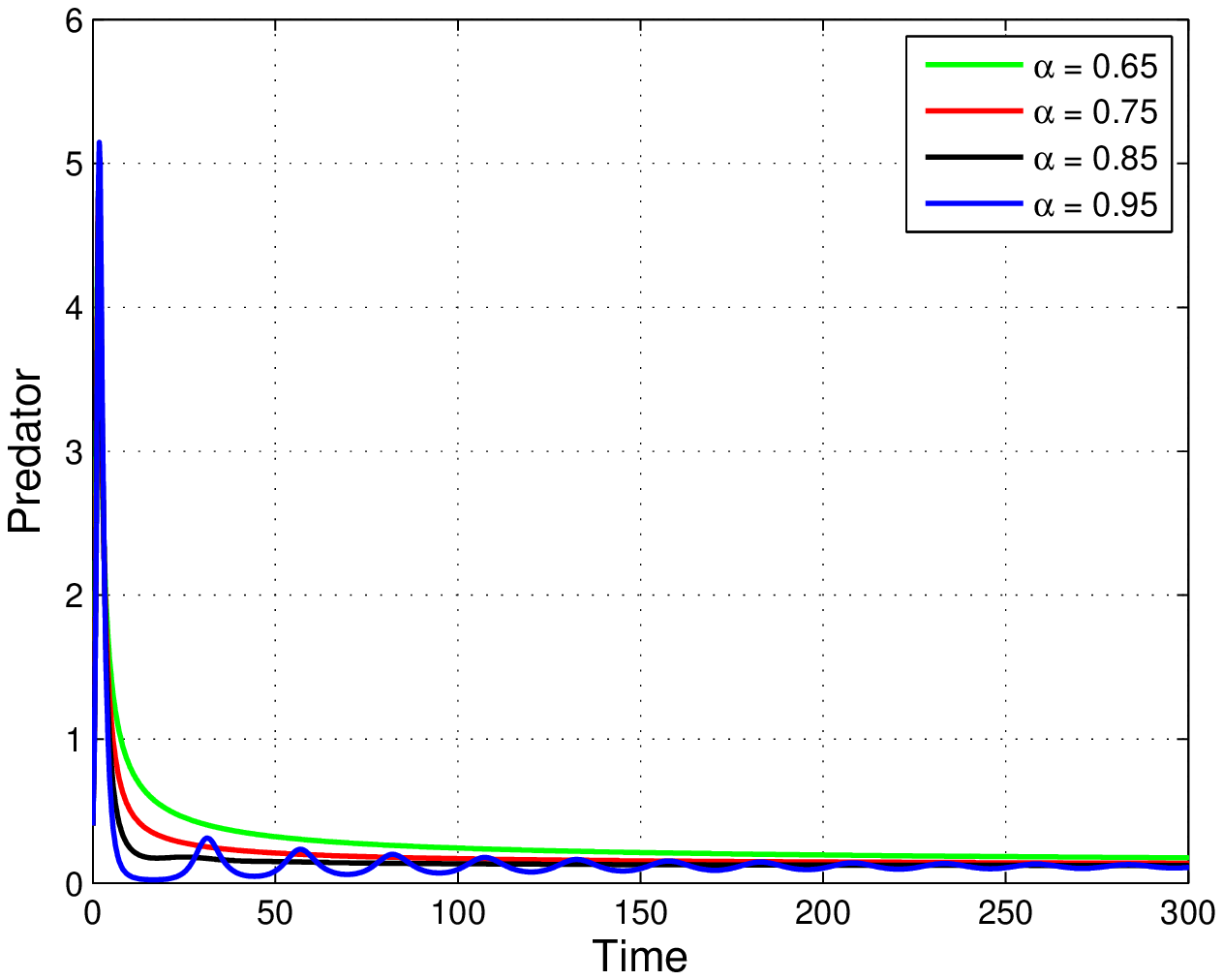}\\
        \footnotesize{(b) Density of prey} & \footnotesize{(c) Density of predator}\\
    \end{tabular}}
    \caption{Numerical simulations for different values of order $\alpha$.}
\end{figure}
\end{center}

\subsubsection{Robustness of the NSFD for big time step increment}

We have made some numerical simulations for the parameters studied in the last case of Section \ref{dyn_beh}. The following results show that the NSFD scheme gives very good agreement with the expected behavior even for very big values of the time step $h$ and a very long integration time. The set of parameters is: $s=0.1$, $K=5$, $q=1$, $q_1=2$, $\beta=15$, $s_0=0.7$, $E=0.3$, $R_0=6.8182$. We take as initial conditions $(0.5,0.4), (0.4,0.1), (0.5,1), (0.3,5)$.

\begin{center}
\begin{figure}[!h]
\centerline{%
    \begin{tabular}{cc}
        \includegraphics[width=0.48\textwidth]{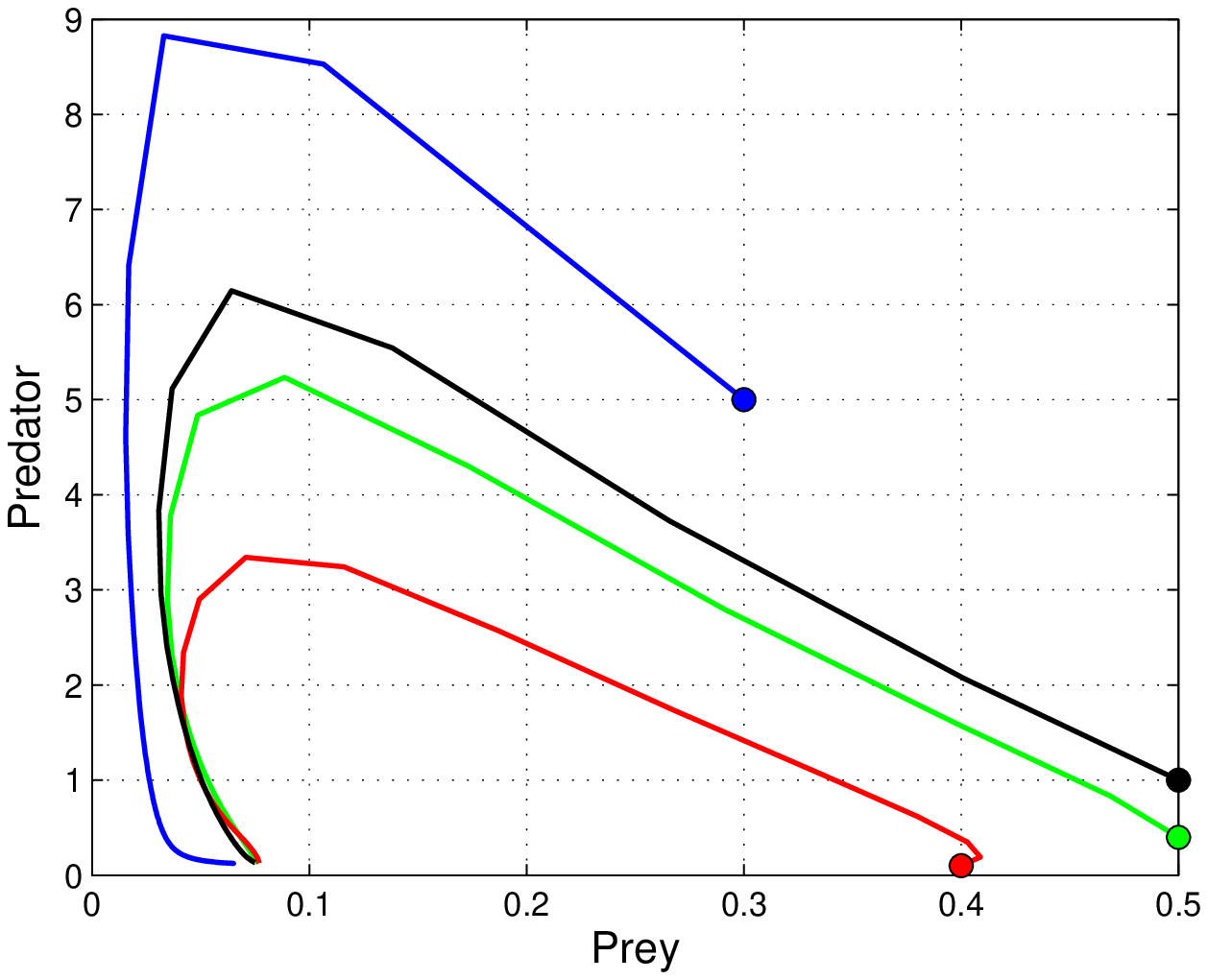} & \includegraphics[width=0.48\textwidth]{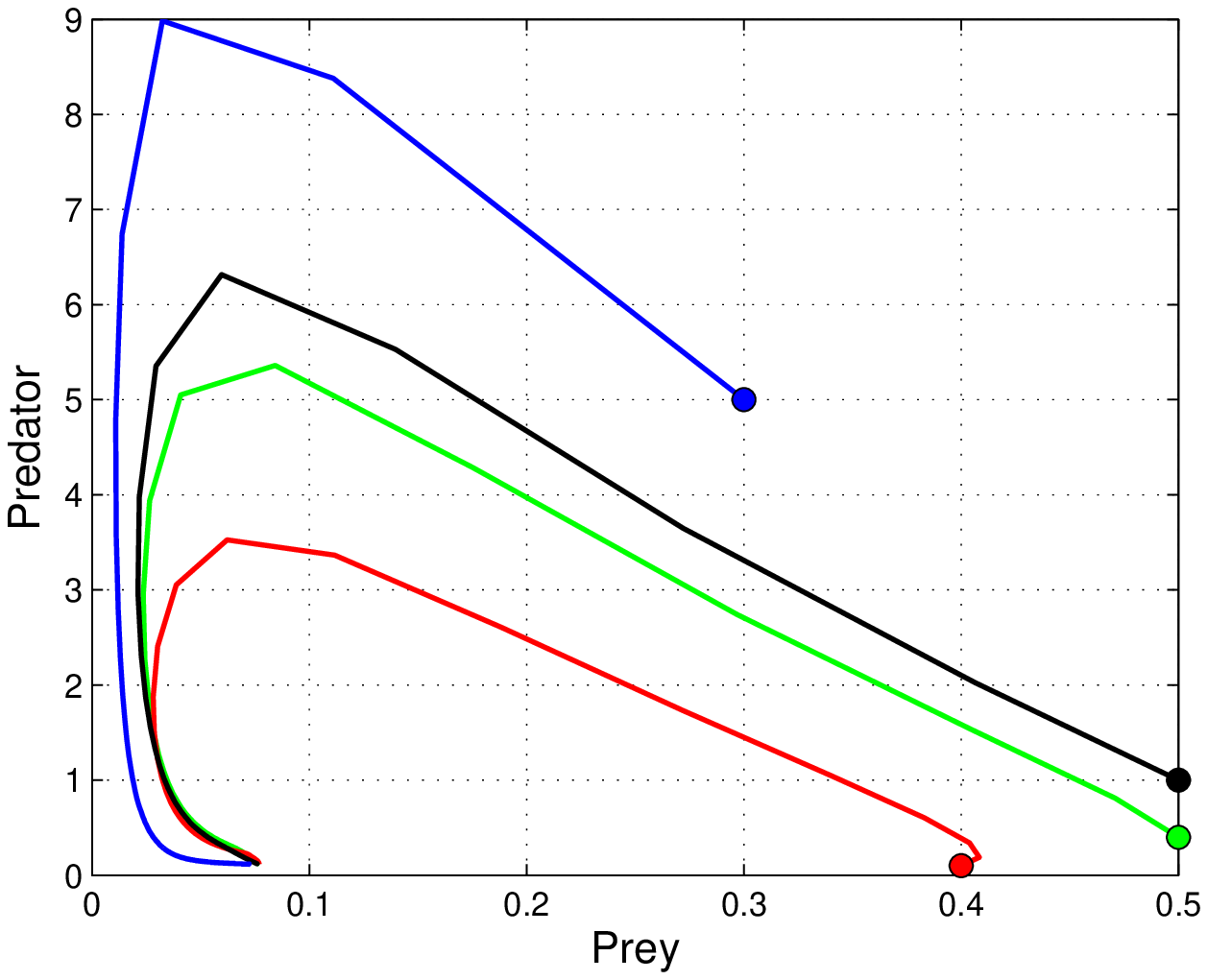}\\
        \footnotesize{(a) $\alpha=0.65$} & \footnotesize{(b) $\alpha=0.75$}\\
    \end{tabular}}
%
\centerline{%
    \begin{tabular}{cc}
        \includegraphics[width=0.48\textwidth]{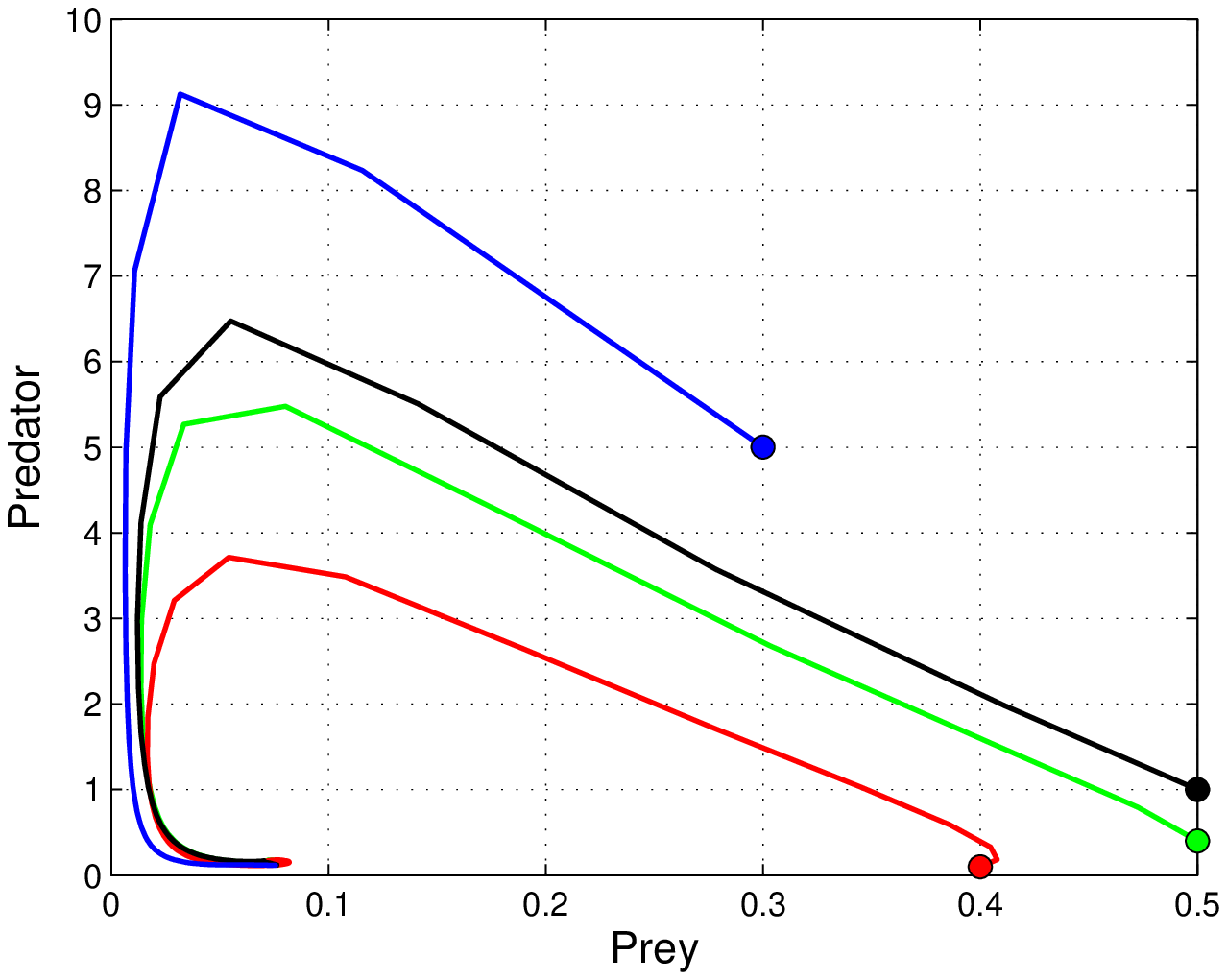} & \includegraphics[width=0.48\textwidth]{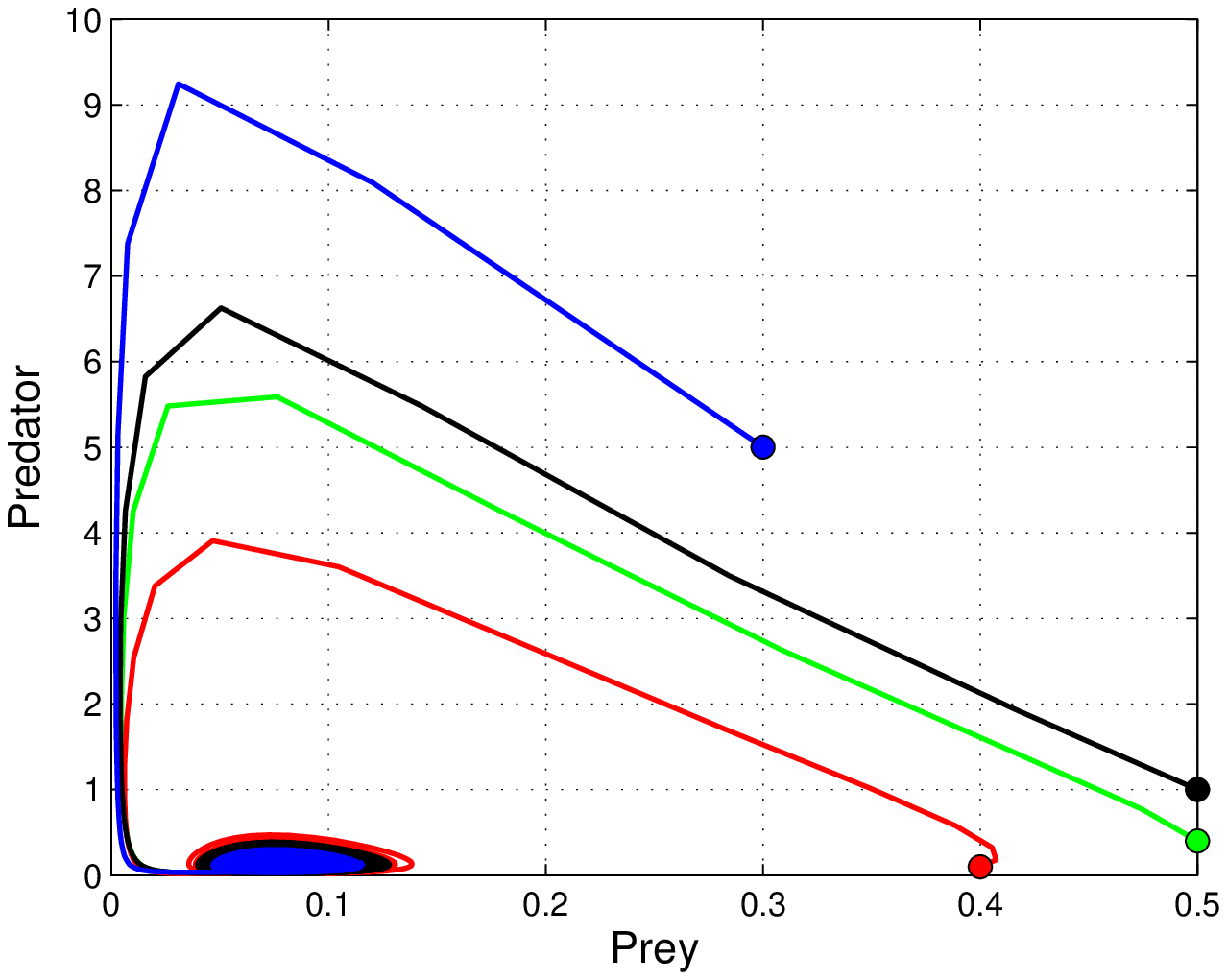}\\
        \footnotesize{(c) $\alpha=0.85$} & \footnotesize{(d) $\alpha=0.95$}\\
    \end{tabular}}
    \caption{Phase portraits for four cases of order $\alpha$ with $h=0.5$ and time integration $T=1500$.}
\end{figure}
\end{center}

Jacky Cresson (*,**) and Anna Szafra\'{n}ska (***)

\begin{small}
(*)  Laboratoire de Math{\'e}matiques Appliquées de Pau, UMR CNRS 5142,

Université de Pau et des Pays de l’Adour,

avenue de l’Université, BP 1155, 64013 Pau Cedex, France.
\vskip 1mm
(*) SYRTE, UMR CNRS 8630,

Observatoire de Paris, France.
\vskip 1mm
(***) Department of Differential Equations and Mathematics Applications,

Gda\'nsk University of Technology,

G. Narutowicz Street 11/12, 80-233 Gda\'nsk, Poland

E-mail: aszafranska@mif.pg.gda.pl
\end{small}

\medskip
\end{document}